\documentclass[12pt]{article}
\usepackage{amssymb}
\usepackage{color}

\usepackage{amsthm}
\usepackage{amsmath}
\usepackage{graphicx}
\allowdisplaybreaks

 \newtheorem{thm}{Theorem}[section]
 
 \newtheorem{lem}[thm]{Lemma}
 
 \theoremstyle{definition}

 \numberwithin{equation}{section}

\newtheorem{theorem}{Theorem}[section]

\theoremstyle{definition}

\theoremstyle{remark}
\newtheorem{remark}[theorem]{Remark}

\begin{document}
\title{Uniform Bound of the Highest-order Energy of the 2D Incompressible Elastodynamics}
\author{Yuan Cai
\footnote{Department of Mathematics, The Hong Kong University of Science and Technology, ClearWater Bay, Kowloon, Hong Kong. {
 Email: maycai@ust.hk}}
}
\date{}
\maketitle

\begin{abstract}
This paper concerns the time growth of the highest-order energy of the
 systems of incompressible isotropic Hookean elastodynamics in two space dimensions.
The global well-posedness of smooth solutions near equilibrium was first obtained
by Lei \cite{Lei16} where the highest-order
generalized energy may have certain growth in time.
We improve the result and show that the highest-order generalized energy is
uniformly bounded for all the time.
The two dimensional incompressible isotropic elastodynamics is a system of nonlocal quasilinear wave equations
where the unknowns decay as $\langle t\rangle^{-\frac12}$.
This suggests the problem is supercritical in the sense that the decay rate is far from integrable.
Surprisingly, we showed that in the highest-order energy estimate, the temporal decay can be strongly enhanced to be subcritical
$\langle t \rangle^{-\frac54}$.
The analysis is based on
the ghost weight energy method by Alinhac \cite{Alinhac01a}, the inherent strong null structure by Lei \cite{Lei16} and the inherent div-curl structure of the system.

\end{abstract}

\quad{ 2020 Mathematics Subject Classification: 35L72, 35Q74, 74B20.}

\maketitle





\section{Introduction}
The Sobolev norm growth is an important mathematical topic.
In continuum mechanics,
the growth of Sobolev norm quantifies the diffusion of energy to higher and higher frequency modes.
There are many well-known works on this subject. For instance,
Kislev-Sverak showed that the gradient of vorticity exhibits double exponential growth in time for all times
on the two dimensional Euler equations in a disk
 \cite{Sv}.
Colliander-Keel-Staffilani-Takaoka-Tao showed finite time $H^s$-norm ($0<s<1$) growth of the cubic defocusing nonlinear Schr\"odinger equation
on the two dimensional torus \cite{Tao}.
In this paper, we consider the time growth of the highest-order Sobolev norm of the global classical solutions to
the Cauchy problem in the two dimensional incompressible isotropic nonlinear elastodynamics.

For homogeneous, isotropic and hyperelastic materials, the motion of
the elastic fluids in Lagrangian coordinates is determined by the following Lagrangian
functional of flow maps:
\begin{align}\label{LaF}
\mathcal{L}(X; T, \Omega) =&
\int_0^T\int_{\Omega}\big(\frac{1}{2}|\partial_tX(t, y)|^2 -
W(\nabla X(t, y))\\\nonumber & +\ p(t, y)\big[\det\big(\nabla X(t,
y)\big) - 1\big]\big)dydt.
\end{align}
Here  $X(t, y)$ is the flow map,  $W \in C^\infty(GL_2, \mathbb{R}_+)$ is the strain energy
function which depends only on the deformation tensor $\nabla X$,
pressure $p(t, y)$ is a Lagrangian
multiplier which is used to force the flow maps to be
incompressible.
Without loss of generality, we will only study 
 the typical Hookean case where the strain energy functional is given by
\begin{equation*}
W(\nabla X) = \frac{1}{2}|\nabla X|^2.
\end{equation*}
In the Hookean case, the Euler-Lagrangian equation of \eqref{LaF} takes
\begin{equation} \label{Elasticity}
(\partial_t^2 - \Delta) X +(\nabla X)^{-\top}\nabla p=0.
\end{equation}
On the other hand, the incompressibility is equivalent to say that the material density is constant within a fluid parcel. For simplicity, we take the density to be 1.
For any given smooth flow map $X(t, y)$,
there holds the following conservation law of mass:
\begin{equation}\nonumber
\int_{\Omega}dy = \int_{\Omega_t}dX,\quad \Omega_t = \{X(t, y)|y \in
\Omega\}
\end{equation}
for any smooth bounded connected domain $\Omega$.
Clearly, the above equation will imply that
\begin{equation}\label{inc}
\det(\nabla X) = 1.
\end{equation}

The global existence of classical solutions to the Cauchy problem of \eqref{Elasticity}-\eqref{inc}
for small initial displacements was proved by Lei \cite{Lei16}.
By introducing the concept of strong null condition and
observing that the incompressible isotropic elastodynamics inherently satisfies such condition,
Lei proved the global existence of this system where the highest-order energy may have certain growth in time.
Actually, Alinhac conjectured that by analogy with similar problems where such a growth has been proved, that this time growth of the highest-order energy is a true phenomenon \cite{Alinhac10}.
However, we can show that the highest-order energy of the two dimensional incompressible
isotropic elastodynamics is uniformly bounded in all the time.
The two dimensional isotropic incompressible elastodynamics is supercritical in the sense
that the $\langle t \rangle^{-\frac 12}$ temporal decay rate of the unknowns is not integrable.
Even if the system satisfies the strong null condition and using the ghost weight technique,
one can obtain enhanced critical decay rate $\langle t \rangle^{-1}$ in time in the highest-order energy estimate, see \cite{Lei16, LSZ13}.
It's truly surprising that we can finally achieve subcritical  $\langle t \rangle^{-\frac54}$ temporal decay
in the highest-order energy estimate and thus imply the uniform bound of the highest-order energy in all time.



\subsection{A review of related results}
The incompressible elastodynamics is a system of nonlocal quasilinear wave equations.
Let us first review some closely related historical works on quasilinear wave
type equations. For quasilinear wave equations in dimension three, and for small
initial data, John and Klainerman \cite{JohnKlainerman84} obtained the
first non-trivial almost global existence result.
When the spatial dimensions are not bigger than three, the global existence would depend
on two basic assumptions: the initial data should be sufficiently small and the
nonlinearities satisfy a type of null condition \cite{Sideris00}. For nonlinear wave
equations with sufficiently small initial data and the null condition is not satisfied,
the finite time blow-up was shown by John \cite{John81}, Alinhac \cite{Alinhac00b} in 3D, and
by Alinhac \cite{Alinhac99a,Alinhac99b, Alinhac01b} in the two dimensional case.
Under the null condition, the  fundamental works on global solutions for three
dimensional scalar wave equation were obtained by Klainerman \cite{Klainerman86} and
Christodoulou \cite{Christodoulou86} independently. In the more difficult two dimensional case, the global solutions were
obtained by Alinhac \cite{Alinhac01a} under the null condition.

The study of   dynamics of  isotropic, hyperelastic and homogeneous materials has a long history.
For three dimensional compressible elastodynamic  systems  commonly referred as elastic waves in literature,
 John \cite{John88} showed the existence of almost global solutions
for small displacements, see also \cite{KS96} for another proof.
On the other hand, John \cite{John84} proved that a genuine nonlinearity condition leads to
the formations of finite time singularities for spherically symmetric, arbitrarily small but
non-trivial displacements. We refer to \cite{T98} for large displacement singularities. When the
genuine nonlinearity condition is excluded,
Sideris \cite{Sideris96, Sideris00} and also Agemi \cite{Agemi00} showed the global existence under a nonresonance condition which is physically consistent with the system.
With an additional repulsive Poisson term, a global existence was established in \cite{HM17} which
allows a general form for the pressure in both two and three dimensional case.
For the incompressible elastodynamics, the only waves presented in the isotropic systems are shear waves which are linearly degenerate. The global-well-posedness was obtained by Sideris and Thomases in \cite{ST05,
ST07}. We also refer to \cite{ST06} for a  unified treatment, and \cite{LW} for some
improvement on the uniform time-independent bounds on the highest-order energy.

In  the two dimensional case, the proof of long time existence  for the
elastodynamics is more difficult due to the weaker temporal decay.
The first long time existence result is due to Lei, Zhou and Sideris \cite{LSZ13} where the
authors showed the almost global existence
in Eulerian coordinates. By introducing the concept of strong null condition
and observing that such condition is intrinsically satisfied for the system in Lagrangian coordinates,
Lei \cite{Lei16} proved  the global well-posedness by using the energy
method of Klainerman and Alinhac's ghost weight approach. Afterwards,
Wang \cite{W17} gave another proof of this latter result using
space-time resonance method \cite{GMS} and a normal form transformation.

Now let us mention some related works on viscoelasticity,
where the viscosity is present in the momentum equation.
The global well-posedness
near equilibrium was first obtained in \cite{LLZ2005} for the two dimensional
case,
see also \cite{LeiS,LeiLZ08}. This method works
in both two and three dimensional  cases. Lei and Zhou \cite{LZ05} obtained similar results by working
directly on the equations for the deformation tensor through an incompressible limit
process. For many related discussions we refer to  \cite{LeiLZ08}, \cite{CZ2006},
\cite{HX2010,HL16,HW2011,Lei10,Lei14,Lei3,Lin,LZ2008, 
QZ2010,ZF} and the references therein.
In all of these works, the initial data is restricted by the viscosity parameter.
Kessenich \cite{Kessenich} established the global
well-posedness for three dimensional incompressible viscoelastic materials
uniformly in the viscosity and in time. The more difficult two dimensional case was obtained
by the author etc. in \cite{CLLM}.

\subsection{Notations, main theorem and ideas}\label{Equations}

Now let us introduce some notations, then present the main theorem.
At last we will discuss the key ideas of the paper.


Denote
\begin{equation*} 
X(t, y) = y + Y(t, y).
\end{equation*}
Then the Hookean incompressible isotropic elastodynamics \eqref{Elasticity} can be written as follows:
\begin{equation}\label{Elasticity-1}
\partial_t^2Y - \Delta Y = -\nabla p- (\nabla Y)^{ \top} (\partial_t^2 - \Delta) Y .
\end{equation}
The constraint \eqref{inc} becomes
\begin{equation}\label{Struc-1}
\nabla\cdot Y = - \det(\nabla Y)
=-\frac12 \nabla_i Y_j\nabla_i^\perp Y_j^\perp.
\end{equation}
The summation convention over repeated indices will always be assumed
through the whole paper.
Here and what follows, we use the following convention
$$(\nabla Y)_{ij} = \frac{\partial Y^i}{\partial y^j}$$
and the following notations
\begin{equation}\nonumber
\omega = \frac{y}{r},\quad r = |y|,
\quad \omega^\perp =\begin{pmatrix}- \omega_2, \omega_1\end{pmatrix},
\quad Y^\perp =\begin{pmatrix}- Y^2, Y^1\end{pmatrix},
\quad \nabla^\perp =\begin{pmatrix}- \partial_2, \partial_1\end{pmatrix}.
\end{equation}

We will use the vector field approach and the generalized energy method of Klainerman \cite{Klainerman85}.
The rotational operator $\Omega$ and the scaling operator $S$  are defined as follows:
\begin{equation*}
\Omega=\partial_\theta=x^\perp\cdot \nabla,
\quad S=t\partial_t+r\partial_r.
\end{equation*}
In the application, we usually  use the the modified rotational operator and modified scaling operator.
For any  vector $Y$ and scalar $p$, we make the following
conventions:
\begin{equation}\nonumber
\begin{cases}
\widetilde{\Omega} Y \triangleq \partial_\theta Y + AY, \quad
\widetilde{\Omega} p
\triangleq \partial_\theta p,\quad \widetilde{\Omega} Y^j = (\widetilde{\Omega} Y)^j,\\[-4mm]\\
\widetilde{S} Y \triangleq (S - 1) Y, \quad \widetilde{S} p
\triangleq S p,\quad \widetilde{S} Y^j = (\widetilde{S} Y)^j,
\end{cases}
\end{equation}
where $A$ is a matrix  defined as follows:
\begin{equation*}
A = \begin{pmatrix}0 & - 1\\ 1 & 0\end{pmatrix}.
\end{equation*}
Let $\Gamma$ be any operator of the set
$\{\partial_t, \partial_1, \partial_2, \widetilde{\Omega},
\widetilde{S}\}$,
$\Gamma^\alpha=\Gamma_1^{\alpha_1}\Gamma_2^{\alpha_2}\Gamma_3^{\alpha_3}\Gamma_4^{\alpha_4}\Gamma_5^{\alpha_5}$
 for any multi-index $\alpha = \{\alpha_1, \alpha_2, \alpha_3,
\alpha_4, \alpha_5\}^\top \in \mathbb{N}^5$.
Then for any $\alpha\in \mathbb{N}^5$, we have the following commutation properties when applying $\Gamma^\alpha$ to \eqref{Elasticity-1} and \eqref{Struc-1}:
\begin{equation}\label{Elasticity-D}
\partial_t^2\Gamma^\alpha Y - \Delta \Gamma^\alpha Y = -\nabla\Gamma^\alpha p
- \sum_{\beta + \gamma= \alpha}C_\alpha^\beta(\nabla
  \Gamma^\beta Y)^{\top}(\partial_t^2 - \Delta)\Gamma^\gamma Y,
\end{equation}
together with the constraint
\begin{align}\label{Struc-2}
\nabla\cdot \Gamma^\alpha Y
&= \sum_{\beta + \gamma =
\alpha}C_\alpha^\beta\big(\partial_1\Gamma^\beta Y^2\partial_2\Gamma^\gamma Y^1 -
  \partial_1\Gamma^\gamma Y^1\partial_2\Gamma^\beta Y^2\big)\\\nonumber
&=\sum_{\beta + \gamma =\alpha}C_\alpha^\beta(\nabla^\perp\Gamma^\beta Y^2\cdot \nabla\Gamma^\gamma Y^1)\\\nonumber
&=-\frac12 \sum_{\beta + \gamma =\alpha}C_\alpha^\beta \nabla_i\Gamma^\beta Y_j\nabla_i^\perp\Gamma^\gamma Y_j^\perp.
\end{align}
Readers can refer to Section 2 in \cite{Lei16} for the derivation of \eqref{Elasticity-D} and \eqref{Struc-2}.
These two equations are the starting point of all the subsequent estimates.

Throughout this paper, we use $\partial$ for space-time
derivatives:
$$\partial = (\partial_t, \partial_1,
\partial_2).$$
Using $\langle a\rangle$ to denote
$$\langle a\rangle = \sqrt{1 + a^2}$$
and $[a]$ to denote the biggest integer which is no more than $a$:
$$[a] = {\rm the\ biggest\ integer\ which\ is\ no\ more\ than\ a}.$$
We often use the following abbreviations:
$$\|\Gamma^{\leq|\alpha|}f\| = \sum_{|\beta| \leq |\alpha|}\|\Gamma^\beta f\|.$$
We need to use Klainerman's generalized energy which is defined by
\begin{equation}\nonumber
\mathcal{E}_{\kappa} = \sum_{|\alpha| \leq \kappa -
1}\|\partial\Gamma^\alpha Y\|_{L^2}^2.
\end{equation}
We define the following weighted $L^2$ generalized energy which is a combination
of Lei in \cite{Lei16} and of Klainerman-Sideris in \cite{KS96}:
\begin{equation*} 
\mathcal{X}_{\kappa} = \sum_{|\alpha| \leq \kappa - 2}\Big(
\int_{r\leq 2\langle t \rangle}\langle t - r\rangle^2 |\partial^2\Gamma^\alpha  Y|^2dy
+ \int_{r > 2\langle t \rangle}\langle t \rangle^2 |\partial^2\Gamma^\alpha Y|^2dy
+\int_{\mathbb{R}^2}\langle t - r\rangle^2 |\partial\nabla\Gamma^\alpha  Y|^2dy
\Big).
\end{equation*}
In order to treat various singular integral with $\langle r\rangle $ or $\langle r-t\rangle$ factor,
we introduce the weighted $L^2$ norm with zero-order differential operator:
\begin{equation*} 
\mathcal{Z}_{\kappa} = \sum_{|\alpha| \leq \kappa - 2}\Big(\int_{\mathbb{R}^2} \langle t - r\rangle^2 |\partial\nabla P(\nabla)\Gamma^\alpha
Y|^2dy \Big),
\end{equation*}
where $P(\nabla)$ is zero-order differential operator with symbol $P(\xi)$.

To describe the space of the initial data, we introduce
\begin{equation}\nonumber
\Lambda = \{\nabla, r\partial_r, \Omega\},
\end{equation}
and
\begin{equation}\nonumber
H^\kappa_\Lambda =\big\{(f, g)\big| \sum_{|\alpha| \leq \kappa -
1}\big(\|\Lambda^\alpha f\|_{L^2} + \|\nabla\Lambda^\alpha f\|_{L^2}
+ \|\Lambda^\alpha g\|_{L^2}\big)<\infty\big\}
\end{equation}
for $\alpha \in\mathbb{N}^4$.
Then we define
\begin{eqnarray}\nonumber
&&H^\kappa_\Gamma(T) = \big\{Y: [0,T) \to
  \mathbb{R}^2\big|\Gamma^\alpha Y \in L^\infty([0,T);
  L^2(\mathbb{R}^2)),\\[-4mm]\nonumber\\\nonumber
&&\quad \quad \quad\ \ \ \ \ \partial_t\Gamma^\alpha Y,
\nabla\Gamma^\alpha
  Y \in L^\infty([0,T); L^2(\mathbb{R}^2)),\ \ \forall\ |\alpha| \leq \kappa - 1\big\}
\end{eqnarray}
with the norm
\begin{equation}\nonumber
\sup_{0 \leq t < T} \mathcal{E}_\kappa^{1/2}(Y).
\end{equation}

Now we are ready to state the main theorem of this paper as follows:
\begin{thm}\label{GlobalW}
Let $W(F) = \frac{1}{2}|F|^2$ be an isotropic Hookean strain energy
function. Let
$(Y_0, v_0) \in H^{\kappa}_\Lambda$ with $\kappa \ge
7$. Suppose that $Y_0$ satisfies the structural constraint
$\eqref{Struc-1}$ and
\begin{equation}\nonumber
\mathcal{E}_{\kappa}^{\frac{1}{2}}(0) = \sum_{|\alpha| \leq \kappa
- 1}\big(
  \|\nabla\Lambda^{\alpha}Y_0\|_{L^2} +
  \|\Lambda^{\alpha}v_0\|_{L^2}\big) \leq
  \epsilon.
\end{equation}
Then there exists a positive constant $\epsilon_0$ and $C_0$,
depending only on $\kappa$, such that if $\epsilon\leq \epsilon_0$, the system of incompressible
isotropic Hookean
elastodynamics \eqref{Elasticity-1} with following initial data
$$Y(0, y) = Y_0(y),\quad \partial_tY(0, y) = v_0(y)$$ has a unique
global classical solution satisfying
\begin{equation}\nonumber
\mathcal{E}_{\kappa }^{\frac{1}{2}}(t) \leq C_0 \epsilon
\end{equation}
uniformly for all $t\in [0,+\infty)$.
\end{thm}
\begin{remark}
The uniform bound suggests the highest-order energy  $\mathcal{E}_{\kappa }$  will not cascade from low frequency to the high frequency.
\end{remark}
\begin{remark}
The uniform bound of the highest-order energy is proved in Lagrangian coordinates.
Whether one can prove the result directly working under Eulerian coordinate is not clear.
\end{remark}
\begin{remark}
In what follows, we will show the {\it{a priori}} estimate
\begin{equation}\label{apriori}
\mathcal{E}_{\kappa}^\prime(t) \leq C_0\langle
t\rangle^{-\frac54}\mathcal{E}_{\kappa}^{2}(t).
\end{equation}
Actually, we can work harder to obtain
\begin{equation}\nonumber
\mathcal{E}_{\kappa}^\prime(t) \leq C_0\langle
t\rangle^{-\frac32}\mathcal{E}_{\kappa}^{2}(t).
\end{equation}
 The conclusion is the same:
uniform bound of the highest-order energy.
Hence we will not try to achieve this optimal temporal decay.

\end{remark}

The main strategy of the proof is as follows: For initial data
satisfying the constraints in Theorem \ref{GlobalW}, we will prove \eqref{apriori}
which is given in the last section, for some absolute
positive constant $C_0$ depending only on $\kappa$.  Once the above
differential inequality is proved, it is easy to show that the
bound for  $\mathcal{E}_{\kappa}$ given in the theorem can never
be reached, which yields the global existence result and the uniform bound of the highest-order
energy $\mathcal{E}_{\kappa}$, thus
completes the proof of Theorem \ref{GlobalW}.


So from now on our main goal will be to show the  \textit{a
priori} differential inequality \eqref{apriori}.
Since $\epsilon_0$ can be arbitrarily small, we will always assume that
$\mathcal{E}_{\kappa }\leq \delta \ll 1$ in the remaining part of
this paper.

\subsection{Key ideas}
 To best illustrate our ideas and to
make the presentation as simple as it could be, we will only focus on
the typical Hookean case. 
Following from Lei \cite{Lei16},
we first estimate $\| \nabla\Gamma^\alpha p \|_{L^2}$, 
$\|(\partial_t^2-\Delta) \Gamma^\alpha Y \|_{L^2}$
and there holds (see Lemma \ref{SN-1})
\begin{equation}\label{test}
\|\nabla\Gamma^\alpha p\|_{L^2} + \|(\partial_t^2 -
\Delta)\Gamma^\alpha Y\|_{L^2} \lesssim \Pi(\alpha, 2).
\end{equation}
Here
$\Pi(\alpha, 2)$ represents some quadratic nonlinearities.
In \cite{Lei16}, Lei introduced the concept of strong null condition  and
observed that 
$\Pi(\alpha, 2)$ satisfies such condition
which will strongly enhance the temporal decay.
We also refer to \cite{CL, CLM, CLLM} for more discussions and applications of strong null condition.
On the other hand, the quantities on the left hand side of \eqref{test} involve
$(|\alpha| + 2)$-th order derivative of unknowns.
Lei \cite{Lei16} was able to organize such that the right hand side of \eqref{test}
involve only $(|\alpha| + 1)$-th order derivative of unknowns at the price of slower temporal decay.
%
%
 This helps to save one derivative in the highest order energy estimate.
By $\eqref{Struc-1}$, the divergence of the unknowns is quadratic nonlinearities.
Hence the generalized energy is determined by its curl part. Thus in the energy estimate, we can apply curl operator which will delete the pressure.
All these enable us to achieve $\langle t\rangle^{-1}$ critical decay in time in the highest-order energy estimate.

Now let us discuss some more key observations to further improve the temporal decay.
In the highest-order energy estimate, for the lower order derivative nonlinear terms which involve no derivative loss,
we need enhanced temporal decay estimate for the D'Alembertian with one derivative.
By the inherent strong null condition and
inherent div-curl structure in the estimate of D'Alembertian,
this is expected and can be achieved by a bootstrap argument.
For the estimate of the highest order derivative nonlinear terms, we need to use the derivative saving version of \eqref{test}
at the price of slower temporal decay.
To salvage the temporal decay back, 
 we need to take advantage of the ghost weight energy method by Alinhac \cite{Alinhac01a}. However, due to the application of curl operator in the energy estimate,
what we have in hand is only the ghost weight energy for the curl part of the unknowns.
(We refer to Section 5.3 for the details of the argument.)
Hence a direct use of this ghost weight energy is not feasible though the null condition is satisfied.
To deal with this difficulty,
we observe that the divergence of the unknowns is quadratic terms which will further improve
the temporal decay. What makes trouble is 
the curl part
of the unknowns.
%
Hence we use the div-curl decomposition on each unknowns in the nonlinearities.
For the nonlinearities containing divergence part, the cubic nonlinearities will be improved to be quartic or higher nonlinearities. For those containing only curl part, we use the null condition and the ghost weight energy of the curl part to improve the temporal decay. Thus we could obtain subcritical temporal decay in the highest-order energy estimate.

Due to the application of div-curl decomposition,
we need to deal with the commutators of generalized derivatives and Riesz transform.
We also need to treat various singular integral with $\langle r\rangle$ or $\langle r-t\rangle$ factor.
However, $\langle r\rangle^2 $ is not in $\mathcal{A}_p$ class of a zero-order differential operator for $p = 2$ in two space dimension.
To overcome this difficulty,
besides  using the weighted $L^2$ norm of Klainerman-Sideris \cite{KS96} and modified one of Lei \cite{Lei16},
we introduce the weighted $L^2$ norm with zero-order differential operators $\mathcal{Z}_\kappa$.
The trick here is that the second derivative unknowns with  $\langle r-t\rangle$-factor will be controlled by D'Alembertian
with $\langle t\rangle$-factor. Then the weighted $L^2$ estimate of the singular integral can pass through.

The remaining part of this paper is organized as follows:
In Section \ref{lemmas} we record
some weighted Sobolev imbedding inequalities.
We also give the estimate for good derivatives and lay down a preliminary
step for estimating weighted generalized $L^2$ energy.
Then the commutator of generalized derivative and Riesz transform will be computed.
Next we estimate the inherent strong null form in Section \ref{WE},
at the end of that section we give the estimate for the weighted
$L^2$ energy.
Section 4 is devoted to the first derivative estimate of D'Alerbertian.
 In last section, we present the highest-order generalized energy estimate.

\section{Preliminaries}\label{lemmas}

In this section, we first present several weighted $L^\infty-L^2$  estimates. They will be used to prove the decay of solutions in $L^\infty$ norm.
Next we state the estimate for the weighted  $L^2$ norm $\mathcal{X}_{\kappa}$
and the estimate for the good derivatives $\omega\partial_t+\nabla$.
At last, the commutator of the generalized operators and the zero-order differential operator will be studied.


The following lemma gives the decay properties of $L^\infty$
norm of derivative of unknowns. It shows that the $L^\infty$ norm of
derivative of unknowns will decay in time  as $\langle
t\rangle^{- \frac{1}{2}}$.
For its proof, we refer to Lemma 3.1 in \cite{Lei16}.
\begin{lem}\label{DecayEF}
Let $t \geq 4$. Then there holds
\begin{equation}\label{C-1}
\langle r \rangle^{\frac{1}{2}}|\partial\Gamma^\alpha Y| \lesssim
\mathcal{E}_{|\alpha| + 3}^{\frac{1}{2}}.
\end{equation}
Moreover, for $r \leq 2\langle t\rangle /3$, or $r \geq 5\langle
t\rangle/4$, there holds
\begin{equation*} 
\langle t\rangle |\partial\Gamma^\alpha Y| \lesssim
\big(\mathcal{E}_{|\alpha| + 1}^{\frac{1}{2}} +
\mathcal{X}_{|\alpha| + 3}^{\frac{1}{2}}\big)\ln^{\frac{1}{2}}\big(e
+ t\big).
\end{equation*}
For $\langle t\rangle/3 \leq r \leq 5\langle t\rangle/2$, there
holds
\begin{equation*} 
\langle r \rangle^{\frac{1}{2}} \langle t-r\rangle^{\frac{1}{2}}
|\partial\Gamma^\alpha Y| \lesssim \mathcal{E}_{|\alpha| + 2}^{\frac{1}{2}}
+ \mathcal{X}_{|\alpha| + 3}^{\frac{1}{2}}.
\end{equation*}
\end{lem}

Next, we state the decay properties of the second order
derivatives of unknowns under $L^\infty$ norm. The following lemma
shows that away from the light cone, the second derivatives of
unknowns decay in time as $\langle t\rangle^{-1}$. But near the
light cone, the decay rate is only $\langle
t\rangle^{-\frac{1}{2}}$, with an extra factor $\langle t -
r\rangle^{- 1}$. For its proof, we refer to Lemma 3.2 in \cite{Lei16}.
\begin{lem}\label{DecayES}
Let $t \geq 4$. Then for $r \leq 2\langle t\rangle /3$, or $r \geq
5\langle t\rangle/4$, there hold
\begin{equation*} 
\langle t\rangle |\partial^2\Gamma^\alpha Y| \lesssim \mathcal{X}_{|\alpha|
+ 4}^{\frac{1}{2}}.
\end{equation*}
For $r \leq 5\langle t\rangle/2$, there holds
\begin{equation*} 
\langle r \rangle^{\frac{1}{2}}\langle t-r\rangle|\partial^2\Gamma^\alpha
Y| \lesssim \mathcal{X}_{|\alpha| + 4}^{\frac{1}{2}} +
\mathcal{E}_{|\alpha| + 3}^{\frac{1}{2}}.
\end{equation*}
\end{lem}

Now we state the following more exact version of the weighted Soblev-type inequalities.
For its proof, we refer to Lemma 3.3 and Lemma 3.4 in \cite{LSZ13}.
\begin{lem}\label{DecayE}
For all $Y(y)\in H^2(\mathbb{R}^2)$, there holds
\begin{align*}
r|Y(y)|^2&\lesssim \sum_{a=0,1} \Large[ \|\partial_r \Omega^a Y\|_{L^2}^2+\|\Omega^a Y\|_{L^2}^2 \Large],\\
r{\langle t-r\rangle}^2|Y(y)|^2&\lesssim \sum_{a=0,1}\Large[\|\langle t-r\rangle \partial_r \Omega^a Y\|_{L^2}^2
+\|\langle t-r\rangle\Omega^a Y\|_{L^2}^2\Large],\\
r{\langle t-r\rangle}|Y(y)|^2&\lesssim \sum_{a=0,1}\Large[\|\langle t-r\rangle \partial_r \Omega^a Y\|_{L^2}^2
+\|\Omega^a Y\|_{L^2}^2\Large],\\
\langle t\rangle \| Y\|_{L^\infty(r\leq \langle t\rangle/2)}
&\lesssim \sum_{|a|\le2}\|\langle t-r\rangle\partial^a Y\|_{L^2},
\end{align*}
provided the right-hand side is finite.
\end{lem}

The next lemma gives a preliminary estimate of the weighted $L^2$
 energy norm $\mathcal{X}_{\kappa}$ and $\mathcal{Z}_{\kappa}$.
\begin{lem}\label{WE-1}
There holds
\begin{align*}\nonumber
&\mathcal{X}_{2} \lesssim \mathcal{E}_2 +
\langle t\rangle^2 \|(\partial_t^2 - \Delta)Y\|^2_{L^2},\\
&\mathcal{Z}_{2} \lesssim \mathcal{E}_2 +
\langle t\rangle^2 \|(\partial_t^2 - \Delta)Y\|^2_{L^2}.
\end{align*}
\end{lem}

\begin{proof}
For the first inequality, see Lemma 3.3 in \cite{Lei16}
and Lemma 2.3 in \cite{KS96}.
The second inequality follows from the first estimate and the $L^2$ boundedness of the
zero-order differential operator.
\end{proof}


At the end of this section, let us record the estimate of the good derivatives
$\omega \partial_t + \nabla$.
For its proof, we refer to Lemma 3.4 in \cite{Lei16}.
\begin{lem}\label{GoodDeri}
For $\frac{\langle t\rangle}{3} \leq r \leq \frac{5\langle
t\rangle}{2}$, there holds
\begin{equation}\nonumber
\langle t\rangle|(\omega \partial_t + \nabla)\partial Y| \lesssim
|\partial Y| + |\partial\Gamma Y| + t|(\partial_t^2 - \Delta) Y)|.
\end{equation}
\end{lem}

Next we show
that the commutator of generalized operator and zero-order differential operator
is still zero-order differential operator.
\begin{lem}\label{Commu}
There holds
\begin{equation*} 
[\Gamma,Q(\nabla)]=P(\nabla),
\end{equation*}
where $Q(\nabla)$ and $P(\nabla)$ are zero-order differential operators.
\end{lem}
%
\begin{proof}
Firstly, if $\Gamma$ takes derivative $\partial$, the lemma is obvious.
Next, if $\Gamma$ takes rotational operator $\Omega_{jk}=x_j \partial_k-x_k\partial_j$ for $1\leq j<k\leq n$
where $n$ is the space dimension. We write
\begin{align*}
&[\Omega_{jk},Q(\nabla)]u \\
&=(x_j \partial_k-x_k\partial_j )Q(\nabla)u-Q(\nabla)(x_j \partial_k-x_k\partial_j )u \\
&=\frac{1}{(2\pi)^{\frac n2}}\int_{\mathbb{R}^n} e^{i x\cdot \xi}[\partial_{\xi_j}(\xi_k Q(\xi)\hat{u}(\xi))-\partial_{\xi_k}(\xi_j Q(\xi)\hat{u}(\xi))
                         -Q(\xi)(\partial_{\xi_j}(\xi_k \hat{u}(\xi))-\partial_{\xi_k}(\xi_j \hat{u}(\xi))) ] d\xi \\
&=\frac{1}{(2\pi)^{\frac n2}} \int_{\mathbb{R}^n} e^{i x\cdot \xi}[\xi_k \partial_{\xi_j}Q(\xi)- \xi_j \partial_{\xi_k} Q(\xi)] \hat{u}(\xi) d\xi \\
&=Q'(\nabla)u.
\end{align*}
It's easy to see $Q'(\nabla)$ is a zero-order differential operator with symbol
$\xi_k \partial_{\xi_j}Q(\xi)- \xi_j \partial_{\xi_k} Q(\xi)$.

 If $\Gamma$ takes scaling operator, we calculate
\begin{align*}
[S,Q(\nabla)]u
&=(t\partial_t+x\cdot \nabla )Q(\nabla)u-Q(\nabla)(t\partial_t+x\cdot \nabla  )u \\
&=\frac{1}{(2\pi)^{\frac n2}}\int_{\mathbb{R}^n} e^{i x\cdot \xi}[\partial_{\xi_j}(\xi_j Q(\xi)\hat{u})
                         -Q(\xi)\partial_{\xi_j}(\xi_j \hat{u}) ] d\xi \\
&=\frac{1}{(2\pi)^{\frac n2}}\int_{\mathbb{R}^n} e^{i x\cdot \xi}[\xi_j \partial_{\xi_j}Q(\xi)] \hat{u} d\xi \\
&=Q''(\nabla)u.
\end{align*}
We also see that $Q''(\nabla)$ is a zero-order differential operator with symbol $\xi\cdot \partial_{\xi}Q(\xi)$.
Thus the lemma is proved.
\end{proof}

\section{Estimate of the $L^2$ weighted norm $\mathcal{X}_{\kappa}$}\label{WE}

Now we estimate the $L^2$ weighted norm
$\mathcal{X}_{\kappa}$. Following from Lei \cite{Lei16},
we show that the $L^2$ norm of $\nabla\Gamma^\alpha p$ and
$(\partial_t^2 - \Delta)\Gamma^\alpha Y$ which involve the
$(|\alpha| + 2)$-th order derivative of unknowns can be bounded by
certain quantities which  involve only $(|\alpha| + 1)$-th order
derivative of unknowns.
Due to  \eqref{Struc-2},
some more structures play an important role in the energy estimate.

\begin{lem}\label{SN-1}
Suppose $\|\nabla Y\|_{L^\infty} \leq \delta$ for some
absolutely positive constant $\delta < 1$. Then there holds
\begin{equation}\nonumber
\|\nabla\Gamma^\alpha p\|_{L^2} + \|(\partial_t^2 -
\Delta)\Gamma^\alpha Y\|_{L^2} \lesssim \Pi(\alpha, 2)
\end{equation}
provided that $\delta$ is small enough, where
\begin{align}\label{D8}
\Pi(\alpha,2) =
& \sum_{\beta + \gamma = \alpha,\
  \gamma \neq \alpha}\big\||\nabla\Gamma^\beta Y||(\partial_t^2
   - \Delta)\Gamma^\gamma Y|\big\|_{L^2}\\\nonumber
& +\ \sum_{\beta + \gamma = \alpha,\
  |\beta| \geq |\gamma|}\Pi_1 + \sum_{\beta + \gamma = \alpha,\
  |\beta| < |\gamma|}\Pi_2,
\end{align}
and $\Pi_1$ and $\Pi_2$ are given by
\begin{align}
&\Pi_1=\|\nabla (-\Delta)^{-1} \nabla_i
(\partial_t\Gamma^\beta Y_j\partial_t\nabla_i^\perp\Gamma^\gamma Y_j^\perp
 -\nabla_k\Gamma^\beta Y_j\nabla_k\nabla_i^\perp\Gamma^\gamma Y_j^\perp) \|_{L^2},
 \label{D9}\\[-4mm]\nonumber\\
&\Pi_2=  \|\nabla (-\Delta)^{-1} \nabla_i^\perp
(\partial_t\nabla_i\Gamma^\beta Y_j\partial_t\Gamma^\gamma Y_j^\perp
 -\nabla_k\nabla_i\Gamma^\beta Y_j\nabla_k\Gamma^\gamma Y_j^\perp) \|_{L^2}.\label{D10}
\end{align}
\end{lem}
\begin{proof}
By \eqref{Elasticity-D}, one writes the system as follows:
\begin{equation}\nonumber
- \nabla\Gamma^\alpha p = (\partial_t^2 - \Delta)
\Gamma^\alpha Y + \sum_{\beta + \gamma = \alpha}C_\alpha^\beta(\nabla \Gamma^\beta Y)^{\top}(\partial_t^2 -
\Delta)\Gamma^\gamma Y.
\end{equation}
Applying the operator $\nabla(- \Delta)^{-1}\nabla\cdot$  to the above system, one obtains
\begin{align}\nonumber
\nabla\Gamma^\alpha p
=& \sum_{\beta + \gamma = \alpha,\
  \gamma \neq \alpha}\nabla(- \Delta)^{-1}\nabla\cdot[C_\alpha^\beta(\nabla\Gamma^\beta Y)^\top(\partial_t^2
   - \Delta)\Gamma^\gamma Y]\\\nonumber
&+\ \nabla(- \Delta)^{-1}\nabla\cdot[(\nabla Y)^\top(\partial_t^2 -
  \Delta)\Gamma^\alpha Y]\\[-4mm]\nonumber\\\nonumber
&+\ \nabla(- \Delta)^{-1}\nabla\cdot(\partial_t^2 -
  \Delta)\Gamma^\alpha Y.
\end{align}
By the $L^2$ boundedness of the Riesz transform,
one has
\begin{align}\label{D1}
\|\nabla\Gamma^\alpha p\|_{L^2}
\lesssim& \sum_{\beta + \gamma = \alpha,\
  \gamma \neq \alpha}\|(\nabla\Gamma^\beta Y)^\top(\partial_t^2
   - \Delta)\Gamma^\gamma Y)\|_{L^2}\\\nonumber
&+\ \|(\nabla Y)^\top(\partial_t^2 - \Delta)\Gamma^\alpha Y\|_{L^2}
  + \|\nabla(- \Delta)^{-1}\nabla\cdot(\partial_t^2
  - \Delta)\Gamma^\alpha Y\|_{L^2}.
\end{align}
Now we calculate the last term on the right hand side of \eqref{D1}.
Using \eqref{Struc-2}, one has
\begin{align}\label{D2}
\nabla\cdot (\partial_t^2
  - \Delta)\Gamma^\alpha Y
=& -\frac12(\partial_t^2 - \Delta) \sum_{\beta + \gamma =\alpha}C_\alpha^\beta \nabla_i\Gamma^\beta Y_j\nabla_i^\perp\Gamma^\gamma Y_j^\perp\\\nonumber
=& -\frac12 \sum_{\beta + \gamma =\alpha}C_\alpha^\beta \nabla_i(\partial_t^2 - \Delta)\Gamma^\beta Y_j\nabla_i^\perp\Gamma^\gamma Y_j^\perp \\\nonumber
& -\frac12 \sum_{\beta + \gamma =\alpha}C_\alpha^\beta \nabla_i\Gamma^\beta Y_j\nabla_i^\perp(\partial_t^2 - \Delta)\Gamma^\gamma Y_j^\perp \\\nonumber
& -\sum_{\beta + \gamma =\alpha}C_\alpha^\beta \partial_t\nabla_i\Gamma^\beta Y_j\partial_t\nabla_i^\perp\Gamma^\gamma Y_j^\perp \\\nonumber
& +\sum_{\beta + \gamma =\alpha}C_\alpha^\beta \nabla_k\nabla_i\Gamma^\beta Y_j\nabla_k\nabla_i^\perp\Gamma^\gamma Y_j^\perp.
\end{align}
The first two terms on the right hand side of \eqref{D2} involving D'Alembertian
are organized as follows
\begin{equation}\label{D3}
-\frac12 \nabla_i\sum_{\beta + \gamma =\alpha}C_\alpha^\beta(\partial_t^2 - \Delta)\Gamma^\beta Y_j\nabla_i^\perp\Gamma^\gamma Y_j^\perp
-\frac12 \nabla_i^\perp\sum_{\beta + \gamma =\alpha}C_\alpha^\beta \nabla_i\Gamma^\beta Y_j(\partial_t^2 - \Delta)\Gamma^\gamma Y_j^\perp .
\end{equation}
For the last two terms on the right hand
side of \eqref{D2}, they are divided into two groups: $|\beta|\geq |\gamma|$ and  $|\beta|< |\gamma|$.
For the first group, they are written as follows:
\begin{align}\label{D4}
&-\sum_{\beta + \gamma =\alpha,\ |\beta|\geq |\gamma|}C_\alpha^\beta
(\partial_t\nabla_i\Gamma^\beta Y_j\partial_t\nabla_i^\perp\Gamma^\gamma Y_j^\perp
 -\nabla_k\nabla_i\Gamma^\beta Y_j\nabla_k\nabla_i^\perp\Gamma^\gamma Y_j^\perp)\\\nonumber
&=-\nabla_i\sum_{\beta + \gamma =\alpha,\ |\beta|\geq |\gamma|}C_\alpha^\beta
(\partial_t\Gamma^\beta Y_j\partial_t\nabla_i^\perp\Gamma^\gamma Y_j^\perp
 -\nabla_k\Gamma^\beta Y_j\nabla_k\nabla_i^\perp\Gamma^\gamma Y_j^\perp).
\end{align}
The other group of $|\beta|< |\gamma|$ can be organized similarly:
\begin{align}\label{D5}
&-\sum_{\beta + \gamma =\alpha,\ |\beta|< |\gamma|}C_\alpha^\beta
(\partial_t\nabla_i\Gamma^\beta Y_j\partial_t\nabla_i^\perp\Gamma^\gamma Y_j^\perp
 -\nabla_k\nabla_i\Gamma^\beta Y_j\nabla_k\nabla_i^\perp\Gamma^\gamma Y_j^\perp)\\\nonumber
&=-\nabla_i^\perp\sum_{\beta + \gamma =\alpha,\ |\beta|< |\gamma|}C_\alpha^\beta
(\partial_t\nabla_i\Gamma^\beta Y_j\partial_t\Gamma^\gamma Y_j^\perp
 -\nabla_k\nabla_i\Gamma^\beta Y_j\nabla_k\Gamma^\gamma Y_j^\perp).
\end{align}
Inserting \eqref{D3}, \eqref{D4} and \eqref{D5} into
\eqref{D2} to derive that
\begin{align}\label{D6}
&\|\nabla(-\Delta)\nabla\cdot (\partial_t^2
  - \Delta)\Gamma^\alpha Y\|_{L^2}\\[-3mm]\nonumber\\\nonumber
&\lesssim \sum_{\beta + \gamma = \alpha}
\big\|  |\nabla\Gamma^\beta Y| \cdot
|(\partial_t^2 - \Delta)\Gamma^\gamma Y| \big\|_{L^2} \\\nonumber
&\quad +\ \sum_{\beta + \gamma = \alpha,\
  |\beta| \geq |\gamma|}
  \|\nabla (-\Delta)^{-1} \nabla_i
(\partial_t\Gamma^\beta Y_j\partial_t\nabla_i^\perp\Gamma^\gamma Y_j^\perp
 -\nabla_k\Gamma^\beta Y_j\nabla_k\nabla_i^\perp\Gamma^\gamma Y_j^\perp) \|_{L^2} \\\nonumber
&\quad +\ \sum_{\beta + \gamma = \alpha,\
  |\beta| < |\gamma|}
  \|\nabla (-\Delta)^{-1} \nabla_i^\perp
(\partial_t\nabla_i\Gamma^\beta Y_j\partial_t\Gamma^\gamma Y_j^\perp
 -\nabla_k\nabla_i\Gamma^\beta Y_j\nabla_k\Gamma^\gamma Y_j^\perp) \|_{L^2}\\\nonumber
&\leq \|\nabla Y\|_{L^\infty} \big\|(\partial_t^2 - \Delta)\Gamma^\alpha Y\|_{L^2}
+\Pi(\alpha,2).
\end{align}
Here $\Pi(\alpha,2)$ is given in \eqref{D8}-\eqref{D10}.

Now let us insert \eqref{D6} into \eqref{D1} to
derive that
\begin{equation}\label{D7}
\|\nabla\Gamma^\alpha p\|_{L^2} \lesssim \|\nabla
Y\|_{L^\infty}\|(\partial_t^2 - \Delta)\Gamma^\alpha Y\|_{L^2} +
\Pi(\alpha,2).
\end{equation}
Then combined with with \eqref{Elasticity-D}, we obtain
\begin{align}\nonumber
&\|(\partial_t^2 - \Delta)\Gamma^\alpha Y\|_{L^2}\\\nonumber
&\lesssim \|\nabla\Gamma^\alpha p\|_{L^2} + \sum_{\beta + \gamma
  = \alpha}\|(\nabla
  \Gamma^\beta Y)^{\top}(\partial_t^2 - \Delta)\Gamma^\gamma
  Y\|_{L^2}\\\nonumber
&\lesssim \|\nabla Y\|_{L^\infty}
\|(\partial_t^2 - \Delta)\Gamma^\alpha Y\|_{L^2} + \Pi(\alpha,2).
\end{align}
Recall the assumption $\|\nabla Y\|_{L^\infty}\ll 1 $, the first term on the right hand side will be absorbed
by the left hand side. Hence we obtain
\begin{equation} \label{D11}
\|(\partial_t^2 - \Delta)\Gamma^\alpha Y\|_{L^2} \lesssim
\Pi(\alpha,2).
\end{equation}
Inserting
\eqref{D11} into \eqref{D7} gives
\begin{equation}\nonumber
\|\nabla\Gamma^\alpha p\|_{L^2} \lesssim \Pi(\alpha,2).
\end{equation}
This ends the proof of the lemma.
\end{proof}

In the next lemma, we will use Lemma \ref{SN-1} to estimate the main
source of nonlinearities in \eqref{Elasticity-D} by carefully
dealing with the last two terms in \eqref{D8}. On the right hand
sides of those estimates that we are going to prove, we did not gain
derivatives since the both sides are of the same order. However, we
gain temporal decay due to the inherent strong null structure of the system.
\begin{lem}\label{SN-2}
Let $\kappa \geq 7$. There exists $\delta > 0$ such that
if $\mathcal{E}_{\kappa } \leq \delta$, then there holds
\begin{equation}\nonumber
\langle t\rangle\|\nabla\Gamma^{\leq\kappa - 2} p\|_{L^2} + \langle
t\rangle\|(\partial_t^2 - \Delta)\Gamma^{\leq\kappa - 2} Y\|_{L^2}
\lesssim \mathcal{E}_{\kappa}^{\frac{1}{2}}\big(\mathcal{E}_{ \kappa}^{\frac{1}{2}} + \mathcal{X}_{ \kappa}^{\frac{1}{2}}\big).
\end{equation}
\end{lem}
\begin{proof}
In view of Lemma \ref{SN-1}, it suffices to treat 
\eqref{D8} for $|\alpha|\leq \kappa-2$.
We first deal with the integrals away from the light cone. For such purpose,
we choose a radial cutoff function $\varphi \in
C_0^\infty(\mathbb{R}^2)$ which satisfies
\begin{equation}\nonumber
\varphi = \begin{cases}1,\quad {\rm if}\ \frac{3}{4} \leq r \leq \frac{6}{5}\\
0,\quad {\rm if}\ r < \frac{2}{3}\ {\rm or}\ r >
\frac{5}{4}\end{cases},\quad |\nabla\varphi| \leq 100.
\end{equation}
For each fixed $t \geq 4$, let $\varphi^t(y) = \varphi(y/\langle
t\rangle)$. Clearly, one has
$$\varphi^t(y) \equiv 1\ \ {\rm for}\ \frac{3\langle t \rangle}{4} \leq r \leq
\frac{6\langle t \rangle}{5},\quad \varphi^t(y) \equiv 0\ \ {\rm
for}\ r \leq \frac{2\langle t \rangle}{3}\ {\rm or}\ r \geq
\frac{5\langle t \rangle}{4}$$ and $$|\nabla\varphi^t(y)| \leq
100\langle t\rangle^{-1}.$$

It is easy to have the following first step
estimate
\begin{align*}
&\sum_{\beta + \gamma = \alpha,\ |\beta| \geq |\gamma|}
  \|\nabla (-\Delta)^{-1} \nabla_i
\big((1-\varphi^t)(\partial_t\Gamma^\beta Y_j\partial_t\nabla_i^\perp\Gamma^\gamma Y_j^\perp
 -\nabla_k\Gamma^\beta Y_j\nabla_k\nabla_i^\perp\Gamma^\gamma Y_j^\perp) \big) \|_{L^2} \\\nonumber
&+\sum_{\beta + \gamma = \alpha,\
  |\beta| < |\gamma|}
  \|\nabla (-\Delta)^{-1} \nabla_i^\perp
\big((1-\varphi^t)(\partial_t\nabla_i\Gamma^\beta Y_j\partial_t\Gamma^\gamma Y_j^\perp
 -\nabla_k\nabla_i\Gamma^\beta Y_j\nabla_k\Gamma^\gamma Y_j^\perp) \big)\|_{L^2}\\
&\lesssim \sum_{\beta + \gamma = \alpha,\
  |\gamma| \leq [|\alpha|/2]}\|\partial\Gamma^\beta Y\|_{L^2}
  \|1_{{\rm supp}(1 - \varphi^t)}\nabla\partial\Gamma^\gamma Y\|_{L^\infty}.
\end{align*}
Using Lemma \ref{DecayES}, the above is bounded by
\begin{equation}\nonumber
\langle t\rangle^{-1}\mathcal{E}_{ |\alpha| +
1}^{\frac{1}{2}}\mathcal{X}_{[|\alpha|/2] + 4}^{\frac{1}{2}}.
\end{equation}
Hence the estimate \eqref{D8} is improved to be:
\begin{align}\label{F-1}
&\Pi(\alpha, 2) \lesssim \langle
  t\rangle^{-1}\mathcal{E}_{ |\alpha| +
  1}^{\frac{1}{2}}\mathcal{X}_{[|\alpha|/2]
  + 4}^{\frac{1}{2}}\\[-4mm]\nonumber\\\nonumber
&\quad +\ \sum_{\beta + \gamma = \alpha,\
  \gamma \neq \alpha}\big\||\nabla\Gamma^\beta Y||(\partial_t^2
   - \Delta)\Gamma^\gamma Y|\big\|_{L^2}\\\nonumber
&\quad +\sum_{\beta + \gamma = \alpha,\ |\beta| \geq |\gamma|}
\Pi_1(\varphi^t) +
\sum_{\beta + \gamma = \alpha,\  |\beta| < |\gamma|}\Pi_2(\varphi^t)\big),
\end{align}
where
\begin{align*}
& \Pi_1(\varphi^t)=\|\nabla (-\Delta)^{-1} \nabla_i
(\varphi^t\partial_t\Gamma^\beta Y_j\partial_t\nabla_i^\perp\Gamma^\gamma Y_j^\perp
 -\varphi^t\nabla_k\Gamma^\beta Y_j\nabla_k\nabla_i^\perp\Gamma^\gamma Y_j^\perp) \|_{L^2},
 \nonumber\\[-4mm]\\\nonumber
& \Pi_2(\varphi^t)=\|\nabla (-\Delta)^{-1} \nabla_i^\perp
(\varphi^t\partial_t\nabla_i\Gamma^\beta Y_j\partial_t\Gamma^\gamma Y_j^\perp
 -\varphi^t\nabla_k\nabla_i\Gamma^\beta Y_j\nabla_k\Gamma^\gamma Y_j^\perp) \|_{L^2}.
\end{align*}
Due to the similarity of $\Pi_1(\varphi^t)$ and $\Pi_2(\varphi^t)$, we only deal with the first term.

To treat $\Pi_1(\varphi^t)$, the main idea is to expect to obtain maximal decay in time.
Hence we need to rewrite $\Pi_1(\varphi^t)$ in a different form to show the strong null structure:
\begin{align*}
&\nabla_i(\varphi^t\partial_t\Gamma^\beta Y_j\partial_t\nabla_i^\perp\Gamma^\gamma Y_j^\perp
 -\varphi^t\nabla_k\Gamma^\beta Y_j\nabla_k\nabla_i^\perp\Gamma^\gamma Y_j^\perp) \\[-4mm]\\
&=\nabla_i(\varphi^t\partial_t\Gamma^\beta Y_j \omega_k(\omega_k\partial_t+\partial_k)\nabla_i^\perp\Gamma^\gamma Y_j^\perp) \\[-4mm]\\
&\quad-\nabla_i^\perp(\nabla_i(\varphi^t(\omega_k\partial_t+\nabla_k)\Gamma^\beta Y_j)\nabla_k\Gamma^\gamma Y_j^\perp) \\[-4mm]\\
&=\nabla_i(\varphi^t\partial_t\Gamma^\beta Y_j \omega_k(\omega_k\partial_t+\partial_k)\nabla_i^\perp\Gamma^\gamma Y_j^\perp) \\[-4mm]\\
&\quad-\nabla_i^\perp(\varphi^t\nabla_i(\omega_k\partial_t+\nabla_k)\Gamma^\beta Y_j\nabla_k\Gamma^\gamma Y_j^\perp) \\[-4mm]\\
&\quad-\nabla_i^\perp(\nabla_i\varphi^t(\omega_k\partial_t+\nabla_k)\Gamma^\beta Y_j\nabla_k\Gamma^\gamma Y_j^\perp).
\end{align*}
Hence by Lemma \ref{GoodDeri}, $\Pi_1(\varphi^t)$ can be estimated as follows:
\begin{align*}
\Pi_1(\varphi^t)&
=\|\nabla (-\Delta)^{-1} \nabla_i
(\varphi^t\partial_t\Gamma^\beta Y_j\partial_t\nabla_i^\perp\Gamma^\gamma Y_j^\perp
 -\varphi^t\nabla_k\Gamma^\beta Y_j\nabla_k\nabla_i^\perp\Gamma^\gamma Y_j^\perp) \|_{L^2} \\[-4mm]\\
&\lesssim
  \|\varphi^t\partial_t\Gamma^\beta Y_j \omega_k(\omega_k\partial_t+\partial_k)\nabla_i^\perp\Gamma^\gamma Y_j^\perp \|_{L^2} \\[-4mm]\\
&\quad
  +\| \varphi^t\nabla_i(\omega_k\partial_t+\nabla_k)\Gamma^\beta Y_j\nabla_k\Gamma^\gamma Y_j^\perp \|_{L^2} \\[-4mm]\\
&\quad
  +\| \nabla_i\varphi^t(\omega_k\partial_t+\nabla_k)\Gamma^\beta Y_j\nabla_k\Gamma^\gamma Y_j^\perp \|_{L^2}\\[-4mm]\\
&\lesssim
\langle t\rangle^{-1}\|\partial\Gamma^\beta Y\|_{L^2}\big(
  \|1_{{\rm supp}(\varphi^t)} \langle t\rangle (\partial_t^2 - \Delta)
  \Gamma^\gamma Y\|_{L^\infty} +\|1_{{\rm supp}(\varphi^t)} \partial\Gamma^{
  \leq|\gamma| + 1} Y\|_{L^\infty}\big)\\[-4mm]\\
&\quad +\ \langle t\rangle^{-1}\|1_{{\rm supp}(\varphi^t)} \partial\Gamma^\gamma Y\|_{L^\infty}
  \big(\|\langle t\rangle(\partial_t^2- \Delta)\Gamma^\beta Y\|_{L^2} + \|\partial\Gamma^{\leq|\beta| + 1} Y\|_{L^2}\big)\\[-4mm]\\
&\quad +\ \langle t\rangle^{-1}\|1_{{\rm supp}(\varphi^t)} \partial\Gamma^\gamma Y\|_{L^\infty}\|
  \partial\Gamma^\beta Y\|_{L^2}.
\end{align*}
Consequently, note $|\gamma| \leq [|\alpha|/2]$ and using
\eqref{C-1} in Lemma \ref{DecayEF},
one has
\begin{align}\label{F-2}
&\sum_{\beta + \gamma = \alpha,
  |\beta|\geq |\gamma|}\Pi_1(\varphi^t)
  \lesssim \langle
  t\rangle^{-\frac{3}{2}}\mathcal{E}_{
  [|\alpha|/2] + 4}^{\frac{1}{2}}\mathcal{E}_{
  |\alpha| + 2}^{\frac{1}{2}}\\\nonumber
&\quad\quad\quad\quad\quad\quad +\ \langle t
  \rangle^{-\frac{3}{2}}\mathcal{E}_{
  |\alpha| + 1}^{\frac{1}{2}}\|\langle t\rangle (\partial_t^2 - \Delta)
  \Gamma^{\leq[|\alpha|/2] + 2} Y\|_{L^2}\\\nonumber
&\quad\quad\quad\quad\quad\quad +\ \langle
t\rangle^{-\frac{3}{2}}\mathcal{E}_{
  [|\alpha|/2] + 3}^{\frac{1}{2}}
  \|\langle t\rangle(\partial_t^2
  - \Delta)\Gamma^{\leq|\alpha|} Y\|_{L^2}.
\end{align}

It remains to estimate the second line of \eqref{F-1}. Using the Sobolev
imbedding $H^2 \hookrightarrow L^\infty$, we have
\begin{align}\label{F-3}
&\sum_{\beta + \gamma = \alpha,\
  \gamma \neq \alpha}\big\||\nabla\Gamma^\beta Y||(\partial_t^2
   - \Delta)\Gamma^\gamma Y|\big\|_{L^2}\\\nonumber
&\lesssim \sum_{\beta + \gamma = \alpha,\
  |\gamma| \leq [|\alpha|/2]}\|\nabla\Gamma^\beta Y\|_{L^2}\|(\partial_t^2
   - \Delta)\Gamma^\gamma Y\|_{L^\infty}\\\nonumber
&\quad +\ \sum_{\beta + \gamma = \alpha,\
  [|\alpha|/2] < |\gamma| < |\alpha|}\|\nabla\Gamma^\beta Y\|_{L^\infty}\|(\partial_t^2
   - \Delta)\Gamma^\gamma Y\|_{L^2}\\\nonumber
&\lesssim \langle t\rangle^{-1}\mathcal{E}_{
  |\alpha| + 1}^{\frac{1}{2}}\|\langle t\rangle(\partial_t^2
   - \Delta)\Gamma^{\leq[|\alpha|/2] + 2} Y\|_{L^2}\\\nonumber
&\quad +\ \langle t\rangle^{-1}\mathcal{E}_{
  [|\alpha|/2] + 3}^{\frac{1}{2}}\|\langle t\rangle(\partial_t^2
   - \Delta)\Gamma^{\leq|\alpha| - 1} Y\|_{L^2}.
\end{align}
Inserting \eqref{F-2} and \eqref{F-3}  into \eqref{F-1}.
Note that similar estimate also holds for $\Pi_2(\varphi^t)$,
we obtain
that
\begin{align}\label{F-4}
\Pi(\alpha,2) \lesssim
&\langle
  t\rangle^{-1}\mathcal{E}_{ |\alpha| +
  2}^{\frac{1}{2}}\big(\mathcal{E}_{[|\alpha|/2]
  + 4}^{\frac{1}{2}} + \mathcal{X}_{[|\alpha|/2]
  + 4}^{\frac{1}{2}}\big)\\\nonumber
&+\ \langle t\rangle^{-1}\mathcal{E}_{
  |\alpha| + 1}^{\frac{1}{2}}\|\langle t\rangle(\partial_t^2
   - \Delta)\Gamma^{\leq[|\alpha|/2] + 2} Y\|_{L^2}\\\nonumber
&+\ \langle t\rangle^{-1}\mathcal{E}_{
  [|\alpha|/2] + 3}^{\frac{1}{2}}\|\langle t\rangle(\partial_t^2
   - \Delta)\Gamma^{\leq|\alpha|} Y\|_{L^2}.
\end{align}
Now for $\kappa \geq 7$ and  $|\alpha| \leq \kappa - 2$, there holds $[|\alpha|/2] + 4
\leq \kappa $. Hence, one can derive from \eqref{F-4} that
\begin{align}\nonumber
\Pi(\alpha,2) \lesssim
&\langle t\rangle^{- 1}\mathcal{E}_{
  \kappa}^{\frac{1}{2}}\big(\mathcal{E}_{
  \kappa }^{\frac{1}{2}} + \mathcal{X}_{
  \kappa }^{\frac{1}{2}}\big)\\\nonumber
&+\langle t\rangle^{-1}\mathcal{E}_{
  \kappa}^{\frac{1}{2}}\|\langle t\rangle(\partial_t^2
   - \Delta)\Gamma^{\leq\kappa - 2} Y\|_{L^2}.
\end{align}
Combined with Lemma \ref{SN-1} gives the
estimate in the lemma.
\end{proof}

Now we are ready  to show that the weighted $L^2$ norm $\mathcal{X}_{\kappa}$ can
be bounded by the weighted energy $\mathcal{E}_{\kappa }$ under some smallness assumption.
\begin{lem}\label{WE-2}
Let $\kappa \geq 7$. Assume there exists $\delta > 0$ such that
if $\mathcal{E}_{\kappa } \leq \delta$, then there holds
\begin{equation}\nonumber
\mathcal{X}_{\kappa} \lesssim \mathcal{E}_{\kappa}.
\end{equation}
\end{lem}
\begin{proof}
Applying Lemma \ref{WE-1} and Lemma \ref{SN-2}, one has
\begin{align}\nonumber
\mathcal{X}_{\kappa}
&\lesssim \mathcal{E}_{
  \kappa}+ \langle t\rangle^2 \|(\partial_t^2 -\Delta)\Gamma^{\leq\kappa - 2}Y\|^2_{L^2}\\\nonumber
&\lesssim \mathcal{E}_{\kappa}+
  \mathcal{E}_{\kappa}\big(\mathcal{E}_{\kappa } + \mathcal{X}_{\kappa }\big).
\end{align}
Then the lemma follows from the assumption $\mathcal{E}_{\kappa}\ll 1$.
\end{proof}

\begin{lem}\label{WE-3}
Let $\kappa \geq 7$. Assume $\mathcal{E}_{\kappa } \ll 1$, then there holds
\begin{align}\nonumber
\langle t\rangle \|(\partial_t^2 -\Delta)\Gamma^{\leq\kappa - 2}Y\|_{L^2}
\lesssim \mathcal{E}_{\kappa},
\quad
\mathcal{Z}_{\kappa}\lesssim \mathcal{E}_{\kappa}.
\end{align}
\end{lem}
\begin{proof}
By Lemma \ref{SN-2} and Lemma \ref{WE-2}, there holds
\begin{align}\nonumber
\langle t\rangle \|(\partial_t^2 -\Delta)\Gamma^{\leq\kappa - 2}Y\|_{L^2}
\lesssim \mathcal{E}_{\kappa}^{\frac12}\big(\mathcal{E}_{\kappa } + \mathcal{X}_{\kappa }\big)^{\frac12}
\lesssim \mathcal{E}_{\kappa}.
\end{align}
The second inequality follows from Lemma \ref{WE-1} and the first inequality.
\end{proof}

\section{The first derivative estimate of D'Alembertian}
In this section, we show that with one more spatial derivative,
the D'Alembertian has better temporal decay.
In view of Lemma \ref{SN-1}, the estimate of D'Alembertian is carried out with the help of
the estimate of pressure.
Here, similar procedure is applied to the first derivative estimate of the pressure and
D'Alembertian.

\begin{lem}\label{SND-1}
Suppose $\|\nabla Y\|_{L^\infty} \leq \delta$ for some
absolutely positive constant $\delta \ll 1$. Then there holds
\begin{align}\nonumber
\|\nabla^2\Gamma^\alpha p\|_{L^2}
+ \|\nabla(\partial_t^2 -\Delta)\Gamma^\alpha Y\|_{L^2}
\lesssim \Phi(\alpha, 3),
\end{align}
where
\begin{align}\label{F10}
\Phi(\alpha, 3)
=& \sum_{\beta + \gamma = \alpha,\
  \gamma \neq \alpha}\big\||\nabla\Gamma^\beta Y||\nabla(\partial_t^2
   - \Delta)\Gamma^\gamma Y|\big\|_{L^2}\\\nonumber
&+\sum_{\beta + \gamma = \alpha}
 \big\| |\nabla^2\Gamma^\beta Y|\cdot|(\partial_t^2- \Delta)\Gamma^\gamma Y| \big\|_{L^2}
  \\\nonumber
&+\sum_{\beta + \gamma = \alpha}
  \big\| |\partial\nabla\Gamma^\beta Y| \cdot|(\omega_k\partial_t+\nabla_k)\nabla\Gamma^\gamma Y|\big\|_{L^2}.
\end{align}
\end{lem}
\begin{proof}
We begin by rewriting
\eqref{Elasticity-D} as follows:
\begin{align}\nonumber
- \nabla\Gamma^\alpha p = (\partial_t^2 - \Delta)
\Gamma^\alpha Y + \sum_{\beta + \gamma = \alpha}C_\alpha^\beta(\nabla \Gamma^\beta Y)^{\top}(\partial_t^2 -
\Delta)\Gamma^\gamma Y.
\end{align}
Applying the  operator $\nabla^2(- \Delta)^{-1}\nabla\cdot$ to the above system, we obtain
\begin{align}\nonumber
\nabla^2\Gamma^\alpha p
=& \sum_{\beta + \gamma = \alpha,\
  \gamma \neq \alpha}\nabla^2(- \Delta)^{-1}\nabla\cdot[C_\alpha^\beta(\nabla\Gamma^\beta Y)^\top(\partial_t^2
   - \Delta)\Gamma^\gamma Y]\\\nonumber
&+\ \nabla^2(- \Delta)^{-1}\nabla\cdot[(\nabla Y)^\top(\partial_t^2 -
  \Delta)\Gamma^\alpha Y]\\[-4mm]\nonumber\\\nonumber
&+\ \nabla^2(- \Delta)^{-1}\nabla\cdot(\partial_t^2 -
  \Delta)\Gamma^\alpha Y.
\end{align}
Using the boundedness of the Riesz transform in $L^2$ norm,
one has
\begin{align}\label{F1}
\|\nabla^2\Gamma^\alpha p\|_{L^2}
\lesssim
& \sum_{\beta + \gamma = \alpha,\
  \gamma \neq \alpha}\big\||\nabla\Gamma^\beta Y||\nabla(\partial_t^2
   - \Delta)\Gamma^\gamma Y|\big\|_{L^2}\\\nonumber
&+\sum_{\beta + \gamma = \alpha}
 \big\| |\nabla^2\Gamma^\beta Y|\cdot|(\partial_t^2- \Delta)\Gamma^\gamma Y| \big\|_{L^2}
  \\\nonumber
&+\ \|\nabla Y\|_{L^\infty}\|\nabla(\partial_t^2 - \Delta)\Gamma^\alpha Y\|_{L^2}
  + \|\nabla\cdot(\partial_t^2- \Delta)\Gamma^\alpha Y\|_{L^2}.
\end{align}
Now we calculate the last term on the right hand side of \eqref{F1}.
Using \eqref{Struc-2}, we compute that
\begin{align}\label{F2}
\nabla\cdot (\partial_t^2
  - \Delta)\Gamma^\alpha Y
=& -\frac12 \sum_{\beta + \gamma =\alpha}C_\alpha^\beta \nabla_i(\partial_t^2 - \Delta)\Gamma^\beta Y_j\nabla_i^\perp\Gamma^\gamma Y_j^\perp \\\nonumber
& -\frac12 \sum_{\beta + \gamma =\alpha}C_\alpha^\beta \nabla_i\Gamma^\beta Y_j\nabla_i^\perp(\partial_t^2 - \Delta)\Gamma^\gamma Y_j^\perp \\\nonumber
&-\sum_{\beta + \gamma =\alpha}C_\alpha^\beta \partial_t\nabla_i\Gamma^\beta Y_j\partial_t\nabla_i^\perp\Gamma^\gamma Y_j^\perp \\\nonumber
&+\sum_{\beta + \gamma =\alpha}C_\alpha^\beta \nabla_i\nabla_k\Gamma^\beta Y_j\nabla_i^\perp\nabla_k\Gamma^\gamma Y_j^\perp.
\end{align}
The last two terms of \eqref{F2} are organized as follows:
\begin{align}\label{F3}
&-\partial_t\nabla_i\Gamma^\beta Y_j\partial_t\nabla_i^\perp\Gamma^\gamma Y_j^\perp
  +\nabla_k  \nabla_i\Gamma^\beta Y_j  \nabla_k\nabla_i^\perp\Gamma^\gamma Y_j^\perp\\[-4mm]\nonumber\\
&=-\omega_k(\omega_k\partial_t+\nabla_k)\nabla_i\Gamma^\beta Y_j\partial_t\nabla_i^\perp\Gamma^\gamma Y_j^\perp
+\nabla_k  \nabla_i\Gamma^\beta Y_j  (\omega_k\partial_t+\nabla_k)\nabla_i^\perp\Gamma^\gamma Y_j^\perp.\nonumber
\end{align}
Inserting \eqref{F3} into \eqref{F2} gives that
\begin{align}\label{F4}
&\|\nabla\cdot (\partial_t^2 - \Delta)\Gamma^\alpha Y\|_{L^2}\\[-4mm]\nonumber\\\nonumber
&\lesssim \sum_{\beta + \gamma = \alpha}
\big\|  |\nabla\Gamma^\beta Y| \cdot |\nabla(\partial_t^2 - \Delta)\Gamma^\gamma Y|\big\|_{L^2} \\\nonumber
&\quad +\sum_{\beta + \gamma = \alpha}
  \big\| |\partial\nabla\Gamma^\beta Y|\cdot |(\omega_k\partial_t+\nabla_k)\nabla\Gamma^\gamma Y| \big\|_{L^2}.
\end{align}
Now let us insert \eqref{F4} into \eqref{F1} to
derive that
\begin{align}\label{F5}
\|\nabla^2\Gamma^\alpha p\|_{L^2}
\lesssim
\|\nabla Y\|_{L^\infty}\|\nabla(\partial_t^2 - \Delta)\Gamma^\alpha Y\|_{L^2}
+\Phi(\alpha,3),
\end{align}
where $\Phi(\alpha,3)$ is given by \eqref{F10}.
Using \eqref{Elasticity-D} and \eqref{F5}, one has
\begin{align*}
\|\nabla(\partial_t^2-\Delta)\Gamma^\alpha Y\|_{L^2}
\lesssim \|\nabla Y\|_{L^\infty}\|\nabla(\partial_t^2 - \Delta)\Gamma^\alpha Y\|_{L^2}
+\Phi(\alpha,3).
\end{align*}
Recall the assumption that $\|\nabla Y\|_{L^\infty}$ is appropriately smaller
than an absolute positive constant $\delta \ll 1$.
Then the first term on the right hand is absorbed by the left hand side.
This clearly yields
\begin{align*}
\|\nabla(\partial_t^2-\Delta)\Gamma^\alpha Y\|_{L^2}
\lesssim \Phi(\alpha,3).
\end{align*}
Inserting the above into \eqref{F5}, one has
\begin{align*}
\|\nabla^2\Gamma^\alpha p\|_{L^2} \lesssim \Phi(\alpha,3).
\end{align*}
This ends the proof of the lemma.
\end{proof}

Now we show the enhanced temporal decay estimate of
the first derivative of D'Alembertian.
\begin{lem}\label{Box3}
Let $\kappa \geq 7$. Assume $\mathcal{E}_{\kappa } \ll 1$, then there holds
\begin{align*}
\langle t\rangle^{\frac 32}\| \nabla (\partial_t^2-\Delta) \Gamma^{\leq\kappa-3} Y\|_{L^2}
\lesssim \mathcal{E}_{\kappa}.
\end{align*}
\end{lem}
\begin{proof}
In view of Lemma \ref{SND-1},
we need to deal with $\Phi(\alpha, 3)$ for $|\alpha|\leq \kappa-3$.
Recall that
\begin{align}\label{F11}
\Phi(\alpha, 3) =
& \sum_{\beta + \gamma = \alpha,\
  \gamma \neq \alpha}\big\||\nabla\Gamma^\beta Y|\cdot|\nabla(\partial_t^2
   - \Delta)\Gamma^\gamma Y|\big\|_{L^2}\\\nonumber
&+\sum_{\beta + \gamma = \alpha}
 \big\| |\nabla^2\Gamma^\beta Y|\cdot|(\partial_t^2- \Delta)\Gamma^\gamma Y| \big\|_{L^2}
  \\\nonumber
&+\sum_{\beta + \gamma = \alpha}
  \big\| |\partial\nabla\Gamma^\beta Y| \cdot |(\omega_k\partial_t+\nabla_k)\nabla\Gamma^\gamma Y|\big\|_{L^2}.
\end{align}
We first treat the last two terms of \eqref{F11}.
For the integrals away from the light cone, it is easy to have the following first step
estimate
\begin{align}\nonumber
&\sum_{\beta + \gamma = \alpha}
 \| (1-\varphi^t)|\nabla^2\Gamma^\beta Y|\cdot|(\partial_t^2- \Delta)\Gamma^\gamma Y| \big\|_{L^2}
  \\\nonumber
&+\ \sum_{\beta + \gamma = \alpha}
  \big\| (1-\varphi^t)|\partial\nabla\Gamma^\beta Y| \cdot |(\omega_k\partial_t+\nabla_k)\nabla\Gamma^\gamma Y| \big\|_{L^2}\\\nonumber
&\lesssim \sum_{\beta + \gamma = \alpha,\
  |\gamma| \leq [|\alpha|/2]}\|1_{{\rm supp}(1 - \varphi^t)} \partial^2\Gamma^\beta Y\|_{L^2}
  \|1_{{\rm supp}(1 - \varphi^t)}\partial^2\Gamma^\gamma Y\|_{L^\infty},
\end{align}
where $\varphi^t$ is defined in
Lemma \ref{SN-2}.
Using Lemma \ref{DecayES} and Lemma \ref{WE-2}, the above is bounded by
\begin{equation}\nonumber
\langle t\rangle^{-2}\mathcal{X}_{ |\alpha| +
2}^{\frac{1}{2}}\mathcal{X}_{[|\alpha|/2] + 4}^{\frac{1}{2}}\leq \langle t\rangle^{-2}\mathcal{E}_{\kappa}.
\end{equation}
For the integral domain near the light cone,
by the first inequality in  Lemma \ref{DecayE}, Lemma \ref{GoodDeri} and Lemma \ref{WE-3}, they are bounded by
\begin{align}\nonumber
&\sum_{\beta + \gamma = \alpha}
 \| \varphi^t|\partial\nabla\Gamma^\beta Y|\cdot|(\partial_t^2- \Delta)\Gamma^\gamma Y| \big\|_{L^2}
  \\\nonumber
&\quad+\ \sum_{\beta + \gamma = \alpha}
  \langle t\rangle^{-1}\big\|\varphi^t
   |\partial\nabla\Gamma^\beta Y|\cdot
  |\nabla\Gamma^{|\gamma|+1} Y|
  \big\|_{L^2}\\\nonumber
&\lesssim  \sum_{\beta + \gamma = \alpha,\ |\beta|\geq |\gamma|}
  \|\nabla^2\Gamma^\beta Y\|_{L^2}\| \varphi^t(\partial_t^2- \Delta)\Gamma^\gamma Y| \big\|_{L^\infty} \\\nonumber
&\quad+ \sum_{\beta + \gamma = \alpha,\ |\beta|< |\gamma|}
  \| \varphi^t\nabla^2\Gamma^\beta Y\|_{L^\infty}\|(\partial_t^2- \Delta)\Gamma^\gamma Y| \big\|_{L^2} \\\nonumber
&\quad+\langle t\rangle^{-\frac32} \mathcal{E}_{|\alpha|+2}^{\frac12} \mathcal{E}_{[\alpha/2]+4}^{\frac12}
\\[-4mm]\nonumber\\\nonumber
&\lesssim\langle t\rangle^{-\frac32}\mathcal{E}_{\kappa}.
\end{align}
It remains to estimate the first term on the right hand side of \eqref{F11}.
By the Sobolev imbedding, it is bounded by
\begin{align*}
&\sum_{\beta + \gamma = \alpha,\ |\beta|\geq |\gamma|}
  \big\| |\nabla\Gamma^\beta Y|\cdot|\nabla(\partial_t^2- \Delta)\Gamma^\gamma Y|\big\|_{L^2}
  \\\nonumber
&+\sum_{\beta + \gamma = \alpha,\ |\beta|<|\gamma|<|\alpha|}
  \big\| |\nabla\Gamma^\beta Y|\cdot|\nabla(\partial_t^2- \Delta)\Gamma^\gamma Y|\big\|_{L^2}
  \\\nonumber
&\lesssim  \mathcal{E}_{|\alpha|+1}^{\frac12} \| \nabla(\partial_t^2- \Delta)\Gamma^{\leq[\alpha/2]+2}
Y|\big\|_{L^2}
+\mathcal{E}_{[\alpha/2]+3}^{\frac12} \| \nabla(\partial_t^2- \Delta)\Gamma^{\leq|\alpha|-1}
Y|\big\|_{L^2}.
\end{align*}

Now for $\kappa \geq 7$ and  $|\alpha| \leq \kappa - 3$, there holds $[|\alpha|/2] + 4
\leq \kappa $. Hence, we collect all the estimates and derive from \eqref{F11} that
\begin{align*}
\langle t\rangle^{\frac 32}\| \nabla (\partial_t^2-\Delta) \Gamma^{\leq\kappa-3} Y\|_{L^2}
\lesssim
&\mathcal{E}_{\kappa}+\langle t\rangle^{\frac 32}\| \nabla (\partial_t^2-\Delta) \Gamma^{\leq\kappa-3} Y\|_{L^2}\mathcal{E}_{\kappa}^{\frac12}.
\end{align*}
Recall the assumption that $\mathcal{E}_{\kappa}\ll 1$. The last term is absorbed by the
left hand side. Thus the lemma is proved.

\end{proof}

\section{Energy estimate}
In this section, we perform the highest-order energy estimate.
Before then, we prepare some lemmas.

\subsection{Equivalence between energy}
In this subsection, we are going to show that the generalized energy is determined by its curl part.
\begin{lem}
Let $\kappa\geq 7$. Assume $\mathcal{E}_{\kappa}\ll 1$.
There holds the equivalent relation
\begin{align*}
&\mathcal{E}_{\kappa}\sim \sum_{|\alpha|\leq \kappa-1}
\|\partial\nabla(-\Delta)^{-1}\nabla^\perp\cdot \Gamma^\alpha Y \|^2_{L^2}.
\end{align*}
\end{lem}
\begin{proof}
Due to the div-curl decomposition
\begin{align*}
\Gamma^\alpha Y=\nabla(-\Delta)^{-1}\nabla\cdot \Gamma^\alpha Y+\nabla^\perp(-\Delta)^{-1}\nabla^\perp\cdot \Gamma^\alpha Y ,
\end{align*}
it's easy to see that
\begin{align}\label{DC}
\| \partial  \Gamma^\alpha Y\|_{L^2}^2
=\| \partial\nabla(-\Delta)^{-1}\nabla\cdot \Gamma^\alpha Y\|_{L^2}^2
+\| \partial\nabla(-\Delta)^{-1}\nabla^\perp\cdot \Gamma^\alpha Y\|_{L^2}^2.
\end{align}
Here we used the identity
\begin{align*}
\| \partial\nabla(-\Delta)^{-1}\nabla^\perp\cdot \Gamma^\alpha Y\|_{L^2}^2=\| \partial\nabla^\perp(-\Delta)^{-1}\nabla^\perp\cdot \Gamma^\alpha Y\|_{L^2}^2.
\end{align*}
Now let us calculate $\partial \nabla\cdot \Gamma^\alpha Y$. By \eqref{Struc-2}, we write
\begin{align*}
\partial\nabla\cdot \Gamma^\alpha Y
&=\partial\sum_{\beta + \gamma =\alpha}C_\alpha^\beta(\nabla^\perp\Gamma^\beta Y^2\cdot \nabla\Gamma^\gamma Y^1) \\\nonumber
&=\sum_{\beta + \gamma =\alpha}C_\alpha^\beta
(\nabla^\perp\partial\Gamma^\beta Y^2\cdot \nabla\Gamma^\gamma Y^1
+\nabla^\perp\Gamma^\beta Y^2\cdot \nabla\partial \Gamma^\gamma Y^1)\\
&=\sum_{\beta + \gamma =\alpha}C_\alpha^\beta
[\nabla^\perp\cdot(\partial\Gamma^\beta Y^2 \nabla\Gamma^\gamma Y^1)
+ \nabla\cdot(\nabla^\perp\Gamma^\beta Y^2\partial \Gamma^\gamma Y^1)].
\end{align*}
Hence
\begin{align*}
& \| \partial\nabla(-\Delta)^{-1}\nabla\cdot \Gamma^\alpha Y\|_{L^2}^2 \\
&\lesssim \sum_{\beta + \gamma =\alpha}
\|\nabla(-\Delta)^{-1}\nabla^\perp\cdot(\partial\Gamma^\beta Y^2 \nabla\Gamma^\gamma Y^1) \|_{L^2}^2 \\
&\quad +\sum_{\beta + \gamma =\alpha} \| \nabla(-\Delta)^{-1}\nabla\cdot(\nabla^\perp\Gamma^\beta Y^2\partial \Gamma^\gamma Y^1)\|_{L^2}^2\\
&\lesssim \sum_{\beta + \gamma =\alpha}
\|  \partial\Gamma^\beta Y \nabla\Gamma^\gamma Y \|_{L^2}^2
\lesssim \mathcal{E}_{\kappa}^2.
\end{align*}
Inserting the above inequality into \eqref{DC} and then summing over $|\alpha|\leq \kappa-1$ yields
\begin{align*}
\mathcal{E}_{\kappa}
\lesssim \mathcal{E}_{\kappa}^2
+\sum_{|\alpha|\leq \kappa-1} \| \partial\nabla(-\Delta)^{-1}\nabla^\perp\cdot \Gamma^\alpha Y\|_{L^2}^2.
\end{align*}
Note the assumption $\mathcal{E}_{\kappa} \ll 1$. The first term on the right hand side
is absorbed by the left hand side. Hence we obtain
\begin{align*}
\mathcal{E}_{\kappa}
\lesssim
\sum_{|\alpha|\leq \kappa-1} \| \partial\nabla(-\Delta)^{-1}\nabla^\perp\cdot \Gamma^\alpha Y\|_{L^2}^2.
\end{align*}
The reverse bound is an easy consequence of the $L^2$ boundedness of the Riesz transform.
Thus the lemma is proved.
\end{proof}

\subsection{Temporal decay of unknowns with Riesz transform}

\begin{lem}\label{DivE}
Let $\kappa \geq 7$.
For all $|\alpha|\leq \kappa-1$, there holds
\begin{align*}
\| \partial\nabla(-\Delta)^{-1}\nabla\cdot \Gamma^\alpha Y\|_{L^2}
\lesssim \langle t\rangle^{-\frac 12} \mathcal{E}_{\kappa}.
\end{align*}
For all $|\alpha|\leq [(\kappa-1)/2]$, there holds
\begin{align*}
&\| \partial\nabla^2(-\Delta)^{-1} \nabla^\perp\cdot\Gamma^\alpha Y\|_{L^\infty}
\lesssim \langle t\rangle^{-\frac12} \mathcal{E}_{\kappa}^{\frac12}, \\
&\| \partial\nabla^2(-\Delta)^{-1}\nabla\cdot \Gamma^\alpha Y\|_{L^\infty}
\lesssim \langle t\rangle^{-1} \mathcal{E}_{\kappa}.
\end{align*}
\end{lem}
\begin{proof}
To show the first inequality, for $|\alpha|\leq \kappa-1$, we employ \eqref{Struc-2} to calculate that
\begin{align*}
&\partial\nabla(-\Delta)^{-1}\nabla\cdot \Gamma^\alpha Y \\\nonumber
&=\partial\nabla(-\Delta)^{-1}
\sum_{\beta + \gamma =\alpha}C_\alpha^\beta(\nabla^\perp\Gamma^\beta Y^2\cdot \nabla\Gamma^\gamma Y^1)\\\nonumber
&=\nabla(-\Delta)^{-1}\nabla^\perp\cdot
\sum_{\beta + \gamma =\alpha}C_\alpha^\beta(\partial\Gamma^\beta Y^2\nabla\Gamma^\gamma Y^1)\\\nonumber
&\quad+\nabla(-\Delta)^{-1} \nabla\cdot\sum_{\beta + \gamma =\alpha}C_\alpha^\beta(\nabla^\perp\Gamma^\beta Y^2 \partial\Gamma^\gamma Y^1).
\end{align*}
Hence using Lemma \ref{DecayEF} and Lemma
\ref{WE-2}, we have
\begin{align*}
&\| \partial\nabla(-\Delta)^{-1}\nabla\cdot \Gamma^\alpha Y\|_{L^2} \\\nonumber
&\lesssim \sum_{\beta + \gamma =\alpha}
\big\| |\nabla\Gamma^\beta Y| \cdot |\partial\Gamma^\gamma Y| \big\|_{L^2}\\\nonumber
&\lesssim \|\partial \Gamma^{\leq|\alpha|} Y\|_{L^2}  \|\partial \Gamma^{\leq[|\alpha|/2]} Y\|_{L^\infty}
\lesssim \langle t\rangle^{-\frac 12} \mathcal{E}_{\kappa}.
\end{align*}


For the second inequality,
by third inequality in Lemma \ref{DecayE},  Lemma \ref{Commu} and Lemma \ref{WE-3},
there holds
\begin{align*}
&\| \partial\nabla^2(-\Delta)^{-1} \nabla^\perp\cdot\Gamma^\alpha Y\|_{L^\infty} \\
&\lesssim \langle t\rangle^{-\frac12}
\sum_{|a|\leq 1}
\big(\|\langle t-r\rangle\partial_r\Omega^a  \partial\nabla^2(-\Delta)^{-1} \nabla^\perp\cdot\Gamma^\alpha Y \|_{L^2}
+\| \Omega^a  \partial\nabla^2(-\Delta)^{-1} \nabla^\perp\cdot\Gamma^\alpha Y \|_{L^2} \big)\\
&\lesssim\langle t\rangle^{-\frac12} \mathcal{E}_{\kappa}^{\frac12}.
\end{align*}

Finally, we treat the third inequality.
If $r\geq \langle t\rangle/2$, by the first inequality in Lemma \ref{DecayE}, Lemma \ref{Commu} and Lemma \ref{WE-3},
there holds
\begin{align} \label{H3}
&\|\partial\nabla^2(-\Delta)^{-1}\nabla\cdot \Gamma^\alpha Y \|_{L^\infty(r\geq \langle t\rangle/2)} \\\nonumber
&\lesssim \langle t\rangle^{-\frac12} \sum_{a=0}^{1}\sum_{b=0}^{1}
\|\nabla^a \Omega^b \partial\nabla^2(-\Delta)^{-1}\nabla\cdot \Gamma^\alpha Y\|_{L^2}\\\nonumber
&\lesssim \langle t\rangle^{-\frac12} \sum_{a=0}^{1}\sum_{b=0}^{1}
\|\nabla^a \Omega^b \partial \nabla\cdot \Gamma^\alpha Y\|_{L^2}.
\end{align}
For $|\alpha|\leq [(\kappa-1)/2]$, we recall by \eqref{Struc-2} that
\begin{align*}
\nabla\cdot \Gamma^\alpha Y
&=-\frac12 \sum_{\beta + \gamma =\alpha}C_\alpha^\beta \nabla_i\Gamma^\beta Y_j\nabla_i^\perp\Gamma^\gamma Y_j^\perp.
\end{align*}
Hence by Lemma \ref{DecayEF} and Lemma \ref{WE-2}, \eqref{H3} can be further bounded by
\begin{equation*}
\langle t\rangle^{-1} \mathcal{E}_{\kappa}.
\end{equation*}
Otherwise, if $r\leq \langle t\rangle/2$, by the last inequality in Lemma \ref{DecayE},
there holds
\begin{align*}
&\|\partial\nabla^2(-\Delta)^{-1}\nabla\cdot \Gamma^\alpha Y \|_{L^\infty(r\leq \langle t\rangle/2)} \\\nonumber
&\lesssim \langle t\rangle^{-1}
\sum_{|a|\leq 2}\| \langle t-r \rangle\partial\nabla^a\nabla^2(-\Delta)^{-1}\nabla\cdot \Gamma^\alpha Y\|_{L^2}.
\end{align*}
Now, by integration by parts, we derive that
\begin{align*}
&\| \langle t-r \rangle \partial\nabla^a\nabla^2(-\Delta)^{-1}\nabla\cdot \Gamma^\alpha Y\|_{L^2}^2 \\
&=\int \langle t-r \rangle^2 \partial\nabla^a\nabla_i\nabla_j(-\Delta)^{-1}\nabla\cdot \Gamma^\alpha Y
\cdot \partial\nabla^a\nabla_i\nabla_j(-\Delta)^{-1}\nabla\cdot \Gamma^\alpha Ydy\\
&=-\int \nabla_i\langle t-r \rangle^2 \partial\nabla^a\nabla_i\nabla_j(-\Delta)^{-1}\nabla\cdot \Gamma^\alpha Y
\cdot \partial\nabla^a\nabla_j(-\Delta)^{-1}\nabla\cdot \Gamma^\alpha Ydy\\
&\quad+\int \langle t-r \rangle^2 \partial\nabla^a\nabla_j\nabla\cdot \Gamma^\alpha Y
\cdot \partial\nabla^a\nabla_j(-\Delta)^{-1}\nabla\cdot \Gamma^\alpha Ydy\\
&=-\int \nabla_i\langle t-r \rangle^2 \partial\nabla^a\nabla_i\nabla_j(-\Delta)^{-1}\nabla\cdot \Gamma^\alpha Y
\cdot \partial\nabla^a\nabla_j(-\Delta)^{-1}\nabla\cdot \Gamma^\alpha Ydy\\
&\quad-\int \nabla_j\langle t-r \rangle^2 \partial\nabla^a\nabla\cdot \Gamma^\alpha Y
\cdot \partial\nabla^a\nabla_j(-\Delta)^{-1}\nabla\cdot \Gamma^\alpha Ydy\\
&\quad+\int \langle t-r \rangle^2 \partial\nabla^a\nabla\cdot \Gamma^\alpha Y
\cdot \partial\nabla^a\nabla\cdot \Gamma^\alpha Ydy\\
&\lesssim \| \langle t-r \rangle \partial\nabla^a\nabla_i\nabla_j(-\Delta)^{-1}\nabla\cdot \Gamma^\alpha Y\|_{L^2}
\|\partial\nabla^a\nabla_j(-\Delta)^{-1}\nabla\cdot \Gamma^\alpha Y\|_{L^2}\\[-4mm]\nonumber\\
&\quad+\|\langle t-r \rangle \partial\nabla^a\nabla\cdot \Gamma^\alpha Y\|_{L^2}
\cdot \|\partial\nabla^a\nabla_j(-\Delta)^{-1}\nabla\cdot \Gamma^\alpha Y\|_{L^2}\\[-4mm]\nonumber\\
&\quad +\|\langle t-r \rangle \partial\nabla^a\nabla\cdot \Gamma^\alpha Y\|^2_{L^2}.
\end{align*}
This further implies that
\begin{align}\label{H4}
&\| \langle t-r \rangle \partial\nabla^a\nabla^2(-\Delta)^{-1}\nabla\cdot \Gamma^\alpha Y\|_{L^2}^2 \\[-4mm]\nonumber\\\nonumber
&\lesssim
\|\partial\nabla^a\nabla(-\Delta)^{-1}\nabla\cdot \Gamma^\alpha Y\|^2_{L^2}
+\|\langle t-r \rangle \partial\nabla^a\nabla\cdot \Gamma^\alpha Y\|^2_{L^2}.
\end{align}
Recall that
\begin{align*}
\partial\nabla\cdot \Gamma^\alpha Y
&=\sum_{\beta + \gamma =\alpha}C_\alpha^\beta
[\nabla^\perp\cdot(\partial\Gamma^\beta Y^2 \nabla\Gamma^\gamma Y^1)
+ \nabla\cdot(\nabla^\perp\Gamma^\beta Y^2\partial \Gamma^\gamma Y^1)].
\end{align*}
Then by the Sobolev imbedding and Lemma \ref{WE-2}, \eqref{H4} can be updated as follows
\begin{equation*}
\| \langle t-r \rangle \partial\nabla^a\nabla^2(-\Delta)^{-1}\nabla\cdot \Gamma^\alpha Y\|_{L^2}^2
\leq \mathcal{E}_{\kappa}^2.
\end{equation*}
%
Thus the third inequality is proved. This ends the proof of the lemma.
\end{proof}

\subsection{Highest-order energy estimate}\label{LEE}
\begin{proof}[Proof of the Theorem \ref{GlobalW}]
Let $\kappa \geq 7$, $|\alpha| \leq \kappa - 1$, $\sigma = t
- r$, $q(\sigma) = \arctan\sigma$. Then $e^{- q(\sigma)}\sim 1$.
Applying the operator $\nabla(- \Delta)^{- 1}\nabla^\perp\cdot$ to \eqref{Elasticity-D},
we have
\begin{align}\label{Elasticity-LL}
&(\partial_t^2 - \Delta)\nabla(- \Delta)^{- 1}\nabla^\perp\cdot
  \Gamma^\alpha Y\\[-4mm]\nonumber\\\nonumber
&=  \sum_{\beta + \gamma = \alpha}
  \nabla(- \Delta)^{- 1}\nabla^\perp\cdot\big[C_\alpha^\beta(\nabla \Gamma^\beta
  Y)^{\top}(\partial_t^2 - \Delta)\Gamma^\gamma Y\big].
\end{align}
Multiplying \eqref{Elasticity-LL} by $e^{- q(\sigma)} \partial_t\nabla(- \Delta)^{-1}\nabla^\perp\cdot
  \Gamma^\alpha Y$ and then integrating over $\mathbb{R}^2$, one has
\begin{align}\label{G0}
&\frac{1}{2}\frac{d}{dt}\int\big(\big|\partial_t\nabla(- \Delta)^{-1}\nabla^\perp\cdot\Gamma^\alpha Y\big|^2
  + \big|\nabla\nabla(- \Delta)^{-1}\nabla^\perp\cdot\Gamma^\alpha
  Y\big|^2\big)e^{- q(\sigma)} dy\\\nonumber
&+\frac{1}{2}\sum_k\int
\frac{\big|(\omega_k\partial_t+\partial_k)\nabla(- \Delta)^{-1}\nabla^\perp\cdot\Gamma^\alpha Y\big|^2}
{\langle t-r\rangle^2}e^{- q(\sigma)} dy\\\nonumber
&\lesssim \sum_{\beta + \gamma = \alpha}
  \big\|\partial_t\nabla(- \Delta)^{-1}\nabla^\perp\cdot
  \Gamma^\alpha Y\big\|_{L^2}
  \big\|\nabla(- \Delta)^{- 1}\nabla^\perp\cdot\big[(\nabla \Gamma^\beta
  Y)^{\top}(\partial_t^2 - \Delta)\Gamma^\gamma Y\big]\big\|_{L^2}\\\nonumber
&\lesssim \mathcal{E}_{\kappa}^{\frac{1}{2}}\sum_{\beta +
  \gamma = \alpha}\|(\nabla
  \Gamma^\beta Y)^{\top}(\partial_t^2 - \Delta)\Gamma^\gamma
  Y\|_{L^2}.
\end{align}

For the expression on the right hand side of \eqref{G0},
we first deal with the lower order terms corresponding to $\gamma\neq\alpha$:
\begin{align}\label{G1}
&\mathcal{E}_{\kappa}^{\frac{1}{2}}\sum_{\beta + \gamma
  = \alpha,\ \gamma \neq \alpha}\|(\nabla
  \Gamma^\beta Y)^{\top}(\partial_t^2 - \Delta)\Gamma^\gamma
  Y\|_{L^2}\\\nonumber
&\lesssim \mathcal{E}_{\kappa}^{\frac{1}{2}}\sum_{\beta + \gamma
  = \alpha,\ |\gamma| \leq [|\alpha|/2]}\|\nabla
  \Gamma^\beta Y\|_{L^2}\|(\partial_t^2 - \Delta)\Gamma^\gamma
  Y\|_{L^\infty}\\\nonumber
&\quad +\ \mathcal{E}_{\kappa}^{\frac{1}{2}}\sum_{\tiny\begin{matrix}\beta + \gamma
  = \alpha, \ \gamma \neq \alpha \\  |\beta| \leq [|\alpha|/2]\end{matrix}}
  \|\nabla
  \Gamma^\beta Y\|_{L^\infty}\|(\partial_t^2 - \Delta)\Gamma^\gamma
  Y\|_{L^2}.
\end{align}
Note $\kappa\geq 7$. For $|\gamma|\leq [|\alpha|/2]$, there holds $|\gamma|\leq \kappa-4$.
By Lemma \ref{DecayEF}, Lemma \ref{SN-2}, Lemma \ref{WE-2} and Lemma \ref{Box3}, \eqref{G1} can be further bounded by
\begin{align}\label{G2}
&\mathcal{E}_{\kappa}\|(\partial_t^2 -
  \Delta)\Gamma^{\leq\kappa - 4}Y\|_{L^\infty} +
  \langle t\rangle^{-\frac{1}{2}}\mathcal{E}_{\kappa}
  \|(\partial_t^2 -\Delta)\Gamma^{\leq\kappa - 2}Y\|_{L^2}\\[-4mm]\nonumber\\\nonumber
&\lesssim \mathcal{E}_{\kappa}
\|(\partial_t^2 -\Delta)\Gamma^{\leq\kappa - 4}Y\|^{\frac12}_{L^2}\|\nabla^2(\partial_t^2 -\Delta)\Gamma^{\leq\kappa - 4}Y\|^{\frac12}_{L^2}+\langle t\rangle^{-\frac32}\mathcal{E}_{\kappa}^{2}\\[-4mm]\nonumber\\\nonumber
&\lesssim\langle t\rangle^{-\frac54}\mathcal{E}_{\kappa}^{2}
+\langle t\rangle^{-\frac32}\mathcal{E}_{\kappa}^{2}.
\end{align}

Next we treat the highest order term in \eqref{G0} corresponding to $\gamma=\alpha$
which is also the most troublesome term.
By Lemma \ref{DecayEF} and Lemma \ref{WE-2}, we deduce that
\begin{align}\label{G4}
\mathcal{E}_{\kappa }^{\frac{1}{2}}
\| (\nabla Y)^{\top} (\partial_t^2 - \Delta)\Gamma^\alpha Y\|_{L^2}
&\leq
\mathcal{E}_{\kappa }^{\frac{1}{2}}
\|\nabla  Y\|_{L^\infty}
\|(\partial_t^2 - \Delta)\Gamma^\alpha Y\|_{L^2}\\\nonumber
&\lesssim \langle t\rangle^{-\frac12}\mathcal{E}_{\kappa}
\|(\partial_t^2 - \Delta)\Gamma^\alpha Y\|_{L^2}.
\end{align}
It remains to estimate
\begin{align*}
\|(\partial_t^2 - \Delta)\Gamma^\alpha Y\|_{L^2}.
\end{align*}
In terms of Lemma \ref{SN-1}, we need to treat
\begin{align*}
\Pi(\alpha,2) =
& \sum_{\beta + \gamma = \alpha,\
  \gamma \neq \alpha}\big\||\nabla\Gamma^\beta Y||(\partial_t^2
   - \Delta)\Gamma^\gamma Y|\big\|_{L^2}\\\nonumber
&+\ \sum_{\beta + \gamma = \alpha,\
  |\beta| \geq |\gamma|}\Pi_1 + \sum_{\beta + \gamma = \alpha,\
  |\beta| < |\gamma|}\Pi_2.
\end{align*}
As in \eqref{G1} and \eqref{G2}, one has
\begin{equation*}
\mathcal{E}_{\kappa}^{\frac12}\sum_{\beta + \gamma = \alpha,\
  \gamma \neq \alpha}\big\||\nabla\Gamma^\beta Y||(\partial_t^2
   - \Delta)\Gamma^\gamma Y|\big\|_{L^2}
\lesssim\langle t\rangle^{-\frac54}\mathcal{E}_{\kappa}^{2}.
\end{equation*}
Consequently, \eqref{G4} are updated as follows:
\begin{align*} 
&\mathcal{E}_{\kappa }^{\frac{1}{2}}
\|\nabla  Y\|_{L^\infty}
\|(\partial_t^2 - \Delta)\Gamma^\alpha Y\|_{L^2}\\\nonumber
&\lesssim \langle t\rangle^{-\frac74}\mathcal{E}_{\kappa}^{\frac52}
+\langle t\rangle^{-\frac12}\mathcal{E}_{\kappa}
\big(\sum_{\beta + \gamma = \alpha,\
  |\beta| \geq |\gamma|}\Pi_1 + \sum_{\beta + \gamma = \alpha,\
  |\beta| < |\gamma|}\Pi_2
\big).
\end{align*}
Hence we are left to  treat $\Pi_1$ and $\Pi_2$. Due to their similarity, we only write down the details for the former one:
\begin{equation}\label{G6}
\sum_{\beta + \gamma =\alpha,\ |\beta|\geq |\gamma|}\Pi_1
\leq
\sum_{\beta + \gamma =\alpha,\ |\beta|\geq |\gamma|}
\|(\partial_t\Gamma^\beta Y_j\partial_t\nabla_i^\perp\Gamma^\gamma Y_j^\perp
 -\nabla_k\Gamma^\beta Y_j\nabla_k\nabla_i^\perp\Gamma^\gamma Y_j^\perp)\|_{L^2}.
\end{equation}
Firstly, we consider the integral domain away from the light cone.
Recall $\kappa\geq 7$, $|\alpha|\leq \kappa-1$, $\beta + \gamma =\alpha$,\ $|\beta|\geq |\gamma|$. Therefore $|\gamma|\leq [\alpha/2]\leq \kappa-4.$
Employing the cutoff function $\varphi^t$ introduced in Lemma \ref{SN-2},
by Lemma \ref{DecayES} and Lemma \ref{WE-2}, there holds
\begin{align*}
&\|(1-\varphi^t)(\partial_t\Gamma^\beta Y_j\partial_t\nabla_i^\perp\Gamma^\gamma Y_j^\perp
 -\nabla_k\Gamma^\beta Y_j\nabla_k\nabla_i^\perp\Gamma^\gamma Y_j^\perp)\|_{L^2} \\[-4mm]\\\nonumber
&\lesssim \|(1-\varphi^t) \partial \Gamma^\beta Y \partial \nabla \Gamma^\gamma Y \|_{L^2} \\[-4mm]\\\nonumber
&\lesssim \langle t\rangle^{-1}\mathcal{E}_{|\beta|+1}^{\frac12} \mathcal{X}_{|\gamma|+4}^{\frac12}
\lesssim \langle t\rangle^{-1}\mathcal{E}_{\kappa}.
\end{align*}
Next we study the case when the integral domain is near the light cone:
\begin{equation}\label{G7}
\|\varphi^t(\partial_t\Gamma^\beta Y_j\partial_t\nabla_i^\perp\Gamma^\gamma Y_j^\perp
 -\nabla_k\Gamma^\beta Y_j\nabla_k\nabla_i^\perp\Gamma^\gamma Y_j^\perp)\|_{L^2}.
 \end{equation}
In order to squeeze extra decay in time, we employ the div-curl decomposition:
\begin{align*}
&Y=-\nabla(-\Delta)^{-1}\nabla\cdot Y-\nabla^\perp(-\Delta)^{-1}\nabla^\perp\cdot Y ,\\
&Y^\perp=-\nabla^\perp (-\Delta)^{-1}\nabla\cdot Y+\nabla(-\Delta)^{-1}\nabla^\perp\cdot Y.
\end{align*}
Here we use the div-curl decomposition on each $\Gamma^\beta Y$ and $\Gamma^\gamma Y$
in \eqref{G7}:
\begin{align*}
&\partial_t\Gamma^\beta Y_j \partial_t\nabla_i^\perp\Gamma^\gamma Y_j^\perp
 -\nabla_k\Gamma^\beta Y_j\nabla_k\nabla_i^\perp\Gamma^\gamma Y_j^\perp \\\nonumber
&=(\partial_t\nabla_j(-\Delta)^{-1}\nabla\cdot \Gamma^\beta Y
+\partial_t\nabla_j^\perp(-\Delta)^{-1}\nabla^\perp\cdot \Gamma^\beta Y) \\\nonumber
&\quad\cdot(\partial_t\nabla_i^\perp\nabla_j^\perp(-\Delta)^{-1}\nabla\cdot \Gamma^\gamma Y
-\partial_t\nabla_i^\perp\nabla_j(-\Delta)^{-1}\nabla^\perp\cdot \Gamma^\gamma Y)
\\\nonumber
&\quad-(\partial_k\nabla_j(-\Delta)^{-1}\nabla\cdot \Gamma^\beta Y
+\partial_k\nabla_j^\perp(-\Delta)^{-1}\nabla^\perp\cdot \Gamma^\beta Y) \\\nonumber
&\qquad\cdot(\partial_k\nabla_i^\perp\nabla_j^\perp(-\Delta)^{-1}\nabla\cdot \Gamma^\gamma Y
-\partial_k\nabla_i^\perp\nabla_j(-\Delta)^{-1}\nabla^\perp\cdot \Gamma^\gamma Y)\\\nonumber
&=\partial_t\Gamma^\beta Y_j\cdot\partial_t\nabla_i^\perp\nabla_j^\perp(-\Delta)^{-1}\nabla\cdot \Gamma^\gamma Y\\\nonumber
&\quad-\partial_t\nabla_j(-\Delta)^{-1}\nabla\cdot \Gamma^\beta Y
\cdot\partial_t\nabla_i^\perp\nabla_j(-\Delta)^{-1}\nabla^\perp\cdot \Gamma^\gamma Y\\\nonumber
&\quad-\partial_t\nabla_j^\perp(-\Delta)^{-1}\nabla^\perp\cdot \Gamma^\beta Y
\cdot\partial_t\nabla_i^\perp\nabla_j(-\Delta)^{-1}\nabla^\perp\cdot \Gamma^\gamma Y\\\nonumber
&\quad-\partial_k\Gamma^\beta Y_j\cdot\partial_k\nabla_i^\perp\nabla_j^\perp(-\Delta)^{-1}\nabla\cdot \Gamma^\gamma Y\\\nonumber
&\quad+\partial_k\nabla_j(-\Delta)^{-1}\nabla\cdot \Gamma^\beta Y
\cdot\partial_k\nabla_i^\perp\nabla_j(-\Delta)^{-1}\nabla^\perp\cdot \Gamma^\gamma Y\\\nonumber
&\quad+\partial_k\nabla_j^\perp(-\Delta)^{-1}\nabla^\perp\cdot \Gamma^\beta Y
\cdot\partial_k\nabla_i^\perp\nabla_j(-\Delta)^{-1}\nabla^\perp\cdot \Gamma^\gamma Y.
\end{align*}
Consequently, \eqref{G7} can be bounded as follows
\begin{align}\label{G8}
&\|\varphi^t(\partial_t\Gamma^\beta Y_j\partial_t\nabla_i^\perp\Gamma^\gamma Y_j^\perp
 -\nabla_k\Gamma^\beta Y_j\nabla_k\nabla_i^\perp\Gamma^\gamma Y_j^\perp)\|_{L^2} \\[-4mm]\nonumber \\ \nonumber
&\lesssim \big\| |\partial
\Gamma^\beta Y|\cdot
 |\partial\nabla^2(-\Delta)^{-1}\nabla\cdot \Gamma^\gamma Y|
\big\|_{L^2}\\[-4mm]\nonumber\\ \nonumber
&\quad +\big\||\partial\nabla(-\Delta)^{-1}\nabla\cdot \Gamma^\beta Y|\cdot
 |\partial\nabla^2(-\Delta)^{-1}\nabla^\perp\cdot \Gamma^\gamma Y|
\big\|_{L^2}\\[-4mm]\nonumber\\ \nonumber
&\quad + \|\varphi^t\big(\partial_t \nabla^\perp_j(-\Delta)^{-1}\nabla^\perp\cdot \Gamma^\beta Y
\partial_t\nabla_i^\perp\nabla_j(-\Delta)^{-1}\nabla^\perp\cdot \Gamma^\gamma Y \\[-4mm]\nonumber\\\nonumber
&\qquad\quad-\nabla_k \nabla^\perp_j(-\Delta)^{-1}\nabla^\perp\cdot \Gamma^\beta Y
\nabla_k\nabla_i^\perp\nabla_j(-\Delta)^{-1}\nabla^\perp\cdot \Gamma^\gamma Y\big)\|_{L^2}.
\end{align}
The first two terms on the right hand side of \eqref{G8} become cubic terms.
By Lemma \ref{DivE}, they are controlled by
\begin{align*}
& \big\| |\partial \Gamma^\beta Y|\cdot
 |\partial\nabla^2(-\Delta)^{-1}\nabla\cdot \Gamma^\gamma Y|
\big\|_{L^2}\\[-4mm]\nonumber\\ \nonumber
&+\big\| |\partial\nabla(-\Delta)^{-1}\nabla\cdot \Gamma^\beta Y|\cdot
 |\partial\nabla^2(-\Delta)^{-1}\nabla^\perp\cdot \Gamma^\gamma Y|
\big\|_{L^2}\\[-4mm]\nonumber\\ \nonumber
&\lesssim
\| \partial \Gamma^\beta Y\|_{L^2}
\|\partial\nabla^2(-\Delta)^{-1}\nabla\cdot \Gamma^\gamma Y\|_{L^\infty}
\\[-4mm]\nonumber\\ \nonumber
&\quad+\| \partial\nabla(-\Delta)^{-1}\nabla\cdot \Gamma^\beta Y\|_{L^2}
 \|\partial\nabla^2(-\Delta)^{-1}\nabla^\perp\cdot \Gamma^\gamma Y\|_{L^\infty}\\ \nonumber
&\lesssim\langle t\rangle^{-1} \mathcal{E}_{\kappa}^{\frac32}+
\langle t\rangle^{-1} \mathcal{E}_{\kappa}^{\frac32} .
\end{align*}
For the last two lines in \eqref{G8},
we refer to the null structure and the ghost weight energy for the curl part.
To show the null structure, they are organized as follows
\begin{align}\label{G9}
&\partial_t \nabla^\perp_j(-\Delta)^{-1}\nabla^\perp\cdot \Gamma^\beta Y
\partial_t\nabla_i^\perp\nabla_j(-\Delta)^{-1}\nabla^\perp\cdot \Gamma^\gamma Y \\\nonumber
&-\nabla_k \nabla^\perp_j(-\Delta)^{-1}\nabla^\perp\cdot \Gamma^\beta Y
\nabla_k\nabla_i^\perp\nabla_j(-\Delta)^{-1}\nabla^\perp\cdot \Gamma^\gamma Y \\\nonumber
&=(\omega_k(\omega_k\partial_t+\partial_k))\nabla^\perp_j(-\Delta)^{-1}\nabla^\perp\cdot \Gamma^\beta Y
\partial_t\nabla_i^\perp\nabla_j(-\Delta)^{-1}\nabla^\perp\cdot \Gamma^\gamma Y\\\nonumber
&\quad-\nabla_k\nabla^\perp_j(-\Delta)^{-1}\nabla^\perp\cdot \Gamma^\beta Y
(\omega_k\partial_t+\partial_k)\nabla_i^\perp\nabla_j(-\Delta)^{-1}\nabla^\perp\cdot \Gamma^\gamma Y.
\end{align}
For the first line on the right hand side of \eqref{G9}, it is estimated by
\begin{align*}
&\| \varphi^t (\omega_k(\omega_k\partial_t+\partial_k))\nabla^\perp_j(-\Delta)^{-1}\nabla^\perp\cdot \Gamma^\beta Y
\partial_t\nabla_i^\perp\nabla_j(-\Delta)^{-1}\nabla^\perp\cdot \Gamma^\gamma Y\|_{L^2}\\\nonumber
&\leq
\big\| \frac{(\omega_k\partial_t+\partial_k)\nabla^\perp_j(-\Delta)^{-1}\nabla^\perp\cdot \Gamma^\beta Y}
{\langle t-r\rangle}
\big\|_{L^2}
\|\varphi^t\langle t-r\rangle\partial_t\nabla_i^\perp\nabla_j(-\Delta)^{-1}\nabla^\perp\cdot \Gamma^\gamma Y\|_{L^\infty}.
\end{align*}
By the second inequality in Lemma \ref{DecayE}, there holds
\begin{align}\label{G10}
&\|\varphi^t\langle t-r\rangle\partial_t\nabla_i^\perp\nabla_j(-\Delta)^{-1}\nabla^\perp\cdot \Gamma^\gamma Y\|_{L^\infty}\\\nonumber
&\lesssim \sum_{a=0, 1}\sum_{ b=0,1}\langle t\rangle^{-\frac12 }
\|\langle t-r\rangle \nabla^a\Omega^b
\partial_t\nabla_i^\perp\nabla_j(-\Delta)^{-1}\nabla^\perp\cdot \Gamma^\gamma Y)\|_{L^2}.
\end{align}
We need to treat the commutators of rotational operator and zero-order differential operators.
By Lemma \ref{Commu} and Lemma \ref{WE-3}, we have
\begin{align*}
&\langle t\rangle^{-\frac12 }\|\langle t-r\rangle\nabla^a\Omega^b\partial_t\nabla_i^\perp\nabla_j(-\Delta)^{-1}\nabla^\perp\cdot \Gamma^\gamma Y\|_{L^2}\\\nonumber
&\lesssim
\langle t\rangle^{-\frac12 }\|\langle t-r\rangle\nabla^a\partial_t\nabla_i^\perp\nabla_j(-\Delta)^{-1}\nabla^\perp\cdot \Omega^b\Gamma^\gamma Y\|_{L^2}\\\nonumber
&\quad+\langle t\rangle^{-\frac12 }\|\langle t-r\rangle\nabla^a\partial_t\nabla (P(\nabla))^b\Gamma^\gamma Y\|_{L^2}
+\langle t\rangle^{-\frac12}\mathcal{E}_{\kappa}^{\frac12}.
\end{align*}
Hence by Lemma \ref{WE-3}, \eqref{G10} can be updated as follows
\begin{align*}
\|\varphi^t\langle t-r\rangle\partial_t\nabla_i^\perp\nabla_j(-\Delta)^{-1}\nabla^\perp\cdot \Gamma^\gamma Y\|_{L^\infty}
\lesssim  \langle t\rangle^{-\frac12}\mathcal{E}_{\kappa}^{\frac12}.
\end{align*}

For the second line on the right hand side of \eqref{G9}, by Lemma \ref{GoodDeri}, Lemma \ref{Commu}
and Lemma \ref{WE-3}, they will be estimated by
\begin{align*}
&\| \varphi^t \nabla_k\nabla^\perp_j(-\Delta)^{-1}\nabla^\perp\cdot \Gamma^\beta Y
(\omega_k\partial_t+\partial_k)\nabla_i^\perp\nabla_j(-\Delta)^{-1}\nabla^\perp\cdot \Gamma^\gamma Y\|_{L^2}\\\nonumber
&\leq
\|  \nabla_k\nabla^\perp_j(-\Delta)^{-1}\nabla^\perp\cdot \Gamma^\beta Y\|_{L^2}
\|\varphi^t (\omega_k\partial_t+\partial_k)\nabla_i^\perp\nabla_j(-\Delta)^{-1}\nabla^\perp\cdot \Gamma^\gamma Y\|_{L^\infty}\\\nonumber
&\lesssim \mathcal{E}_\kappa^{\frac12}
\|\varphi^t (\omega_k\partial_t+\partial_k)\nabla_i^\perp\nabla_j(-\Delta)^{-1}\nabla^\perp\cdot \Gamma^\gamma Y\|_{L^\infty}\\\nonumber
&\lesssim \langle t\rangle^{-1}\mathcal{E}_\kappa^{\frac12}
\|\varphi^t
(|\partial\nabla_j(-\Delta)^{-1}\nabla^\perp\cdot \Gamma^\gamma Y|
+|\partial\Gamma\nabla_j(-\Delta)^{-1}\nabla^\perp\cdot \Gamma^\gamma Y|)
\|_{L^\infty}\\\nonumber
&\quad+ \mathcal{E}_\kappa^{\frac12}
\|\varphi^t (\partial_t^2-\Delta)\nabla_j(-\Delta)^{-1}\nabla^\perp\cdot \Gamma^\gamma Y\|_{L^\infty}\\\nonumber
&\lesssim \langle t\rangle^{-1}\mathcal{E}_\kappa.
\end{align*}
Therefore, \eqref{G6} can be updated as follows
\begin{align*}
\sum_{\beta + \gamma =\alpha,\ |\beta|\geq |\gamma|}\Pi_1
\lesssim \langle t\rangle^{-1}\mathcal{E}_{\kappa}
+\langle t\rangle^{-\frac12}\mathcal{E}_{\kappa}^{\frac12}
\cdot \sum_{|\alpha|\leq \kappa-1}\sum_k\big\|\frac{(\omega_k\partial_t+\partial_k)\nabla(- \Delta)^{-1}\nabla^\perp\cdot\Gamma^\alpha Y}
{\langle t-r\rangle} \big\|_{L^2}.
\end{align*}
Note that the same estimate also holds for
$$\sum_{\beta + \gamma =\alpha,\ |\beta|< |\gamma|}\Pi_2.$$

Finally we collect all the estimates in this section to arrive at
\begin{align*}
&\frac{1}{2}\frac{d}{dt}\int\big|\partial\nabla(- \Delta)^{-1}\nabla^\perp\cdot\Gamma^\alpha Y\big|^2
  e^{- q(\sigma)} dy\\\nonumber
&+\frac{1}{2}\sum_k\int
\frac{\big|(\omega_k\partial_t+\partial_k)\nabla(- \Delta)^{-1}\nabla^\perp\cdot\Gamma^\alpha Y\big|^2}
{\langle t-r\rangle^2}e^{- q(\sigma)} dy\\\nonumber
&\lesssim  \langle t\rangle^{- 5/4}\mathcal{E}_{\kappa}^{2}
 +\langle t\rangle^{-1} \mathcal{E}_{\kappa}
 \sum_{|\alpha|\leq \kappa-1}\sum_k \big\|\frac{(\omega_k\partial_t+\partial_k)\nabla(- \Delta)^{-1}\nabla^\perp\cdot\Gamma^\alpha Y}
{\langle t-r\rangle} \big\|_{L^2} \\\nonumber
&\lesssim
\langle t\rangle^{- 5/4}\mathcal{E}_{\kappa}^{2}
+\langle t\rangle^{- 2}\mathcal{E}_{\kappa}^{2}
 +\delta \sum_{|\alpha|\leq \kappa-1}\sum_k\big\|\frac{(\omega_k\partial_t+\partial_k)\nabla(- \Delta)^{-1}\nabla^\perp\cdot\Gamma^\alpha Y}
{\langle t-r\rangle} \big\|^2_{L^2}
\end{align*}
for any positive $\delta>0$.
Summing over all $|\alpha|\leq \kappa-1$. Then choosing sufficiently small $\delta$, the last term will be absorbed by the left hand side, thus
we arrive at
\begin{equation}\label{G11}
\frac{d}{dt} \sum_{|\alpha|\leq \kappa-1}
\int\big|\partial\nabla(- \Delta)^{-1}\nabla^\perp\cdot\Gamma^\alpha Y\big|^2
  e^{- q(\sigma)} dy
\lesssim \langle t\rangle^{-\frac54}\mathcal{E}_{\kappa}^{2}.
\end{equation}
Then we can replace all $\mathcal{E}_{\kappa}$ appeared
throughout this paper by $\big\|\partial\nabla(- \Delta)^{-1}\nabla^\perp\cdot\Gamma^{\leq \kappa-1} Y\big\|_{L^2}^2$
without changing the final result. Then \eqref{G11} gives the
differential inequality \eqref{apriori} which would finally imply the result.
\end{proof}

\section*{Acknowledgement.}

The author would like to thank Professor Zhen Lei at Fudan University for many stimulating discussions.
The work was partially supported by Hong Kong RGC Grant GRF 16308820.


\end{document}